\documentclass[12pt,a4paper,]{article}
\usepackage[latin1]{inputenc}
\usepackage[T1]{fontenc}
\usepackage[english]{babel}
\usepackage{amsmath,amssymb,amsfonts,amsthm}
\usepackage[pdftex]{graphicx}
\usepackage[pdftex]{color}
\usepackage{enumerate}
\usepackage{dsfont}
\usepackage[all]{xy}
\usepackage{fancybox}
\usepackage{makeidx}
\usepackage{multicol}
\usepackage{mathrsfs}
\usepackage{titlesec}
\usepackage{hyperref}

\newcommand{\rigtharrow}{\rightarrow}

\makeindex

\usepackage{anysize}
\marginsize{2.5cm}{2.4cm}{1.6cm}{2.3cm}

\newcommand{\bs}{ \backslash }

\newtheorem{teo}{Theorem}
\newtheorem{prop}[teo]{Proposition}
\newtheorem{lema}[teo]{Lemma}
\newtheorem{cor}[teo]{Corollary}
\theoremstyle{definition}
\newtheorem{deff}[teo]{Definition}
\newtheorem{exa}[teo]{Example}
\newtheorem{rmk}[teo]{Remark}

\newcommand{\ce}{ \mathcal{E}}

\newcommand{\cep}{{ \mathcal{E}'}}

\newcommand{\ceo}{{ \mathcal{E}_0 }}
\newcommand{\ceop}{ {\mathcal{E}_0' }}

\newcommand{\epsil}{\varepsilon}

\newcommand{\tx}{{\widetilde{X}}}
\newcommand{\txp}{{\widetilde{X}'}}

\newcommand{\ty}{{\widetilde{Y}}}

\newcommand{\hx}{{\widehat{X}}}
\newcommand{\hxp}{{\widehat{X}'}}
\newcommand{\hxpp}{{\widehat{X}''}}
\newcommand{\hy}{{\widehat{Y}}}

\newcommand{\tf}{{\widetilde{f}}}

\newcommand{\tgg}{{\widetilde{g}}}

\newcommand{\hf}{{\widehat{f}}}

\newcommand{\hg}{{\widehat{g}}}

\newcommand{\xid}{{{(X,\widehat {X},\widetilde{X})}}}
\newcommand{\xpid}{{{(X',\widehat {X}',\widetilde{X}')}}}

\newcommand{\xiv}{{{(X,\widehat {X},\widetilde{X})}}}
\newcommand{\xpiv}{{{(X',\widehat {X}',\widetilde{X}')}}}
\newcommand{\xppiv}{{{(X'',\widehat {X}'',\widetilde{X}'')}}}

\newcommand{\Nset}{\mathds{N}}

\newcommand{\Rset}{\mathds{R}}
\newcommand{\Cset}{\mathds{C}}

 \DeclareMathOperator{\intt}{int}

\DeclareMathOperator{\adh}{adh}
\renewcommand{\int}{\intt}
\newcommand{\dd}{{\ddabc}}
\DeclareMathOperator{\ddabc}{d}

\DeclareMathOperator{\BB}{B} \DeclareMathOperator{\BBC}{{\overline
B}}
 \newcommand{\ddd}{{\dddabc}}
 \DeclareMathOperator{\dddabc}{D}

\DeclareMathOperator{\kh}{{h}}

\newcommand{\calt}{{\mathcal T}}

\title{A $C_0$ coarse structure for families of pseudometrics and the Higson-Roe functor}
\author{Jesús P. Moreno-Damas}
\date{}

\begin{document}

\maketitle

\begin{abstract}
This paper deepens into the relations between coarse spaces and
compactifications, by defining a $C_0$ coarse structure attached to
a family of pseudometrics. This definition allow us to give a more
topological point of view on the relations between coarse structures
and compactifications ---like the Higson-Roe compactification,
corona and functor and the topological coarse structure attached to
a compactification---, define new functors and giving new relations
between them, in particular, some equivalences of categories.

\end{abstract}

\section{Introduction}

Starting from \cite{highrc}, in \cite{cg2,cg} Roe develops the
relations between compactifications and coarse spaces, by defining
the topological coarse structure attached to a compactification and
the Higson-Roe compactification attached to a proper coarse
structure. He also define, by an algebraic method, the Higson-Roe
functor, from the coarse spaces to the coronas of their attached
compactifications, called the Higson-Roe coronas. This kind of
relations can be represented in the following diagram:

\begin{equation}\label{jajja1}\xymatrix{\textrm{Coarse structures}\ar^{*}[r]\ar^\nu[rd]&\textrm{Compactifications}\ar[l],\ar^{\partial}[d]\ar[l]\\
&\textrm{Compact spaces}}\end{equation} where $\partial$ is the
corona of a compactification and $\nu$ is the Higson-Roe functor.

In \cite{zs}, the authors, with a topological point of view, focus
in the case in which the spaces are complements of Z-sets of the
Hilbert cube, where the compactification is the Hilbert cube, and
work with continuous maps. They prove that the topological coarse
structure attached to that compactification is the $C_0$ coarse
structure
---defined by Wright in \cite{wright,wright3}--- attached to any
metric of the Hilbert cube. Among other results, they give a
topological point of view of that facts and define an equivalence of
categories. More works have gone in that direction, for example
\cite{niya} and \cite{chig}.

Working in \cite{moreno,moreno2} ---my PhD Thesis, directed by
Morón, one of the authors of \cite{zs}---  we observed that some of
the results of \cite{zs}, can be extended to all compactifications
if we generalize the concept of $C_0$ coarse structure to a family
of pseudometrics. Developing this coarse structure, we have tools to
study the Higson Roe compactification, corona and functor and the
topological coarse structure with a more topological point of view.
We complete the diagram (\ref{jajja1}), by becoming the Higson Roe's
compactification and the topological coarse structure ---represented
by $*$ in the diagram--- into functors, keeping their particular
properties. To do it, we need to define a new category of morphisms,
the asymptotically continuous maps, a kind of maps with involve the
coarse maps between proper coarse spaces and the proper and
continuous maps. Moreover, we give an alternative topological
definition of the Higson-Roe functor, like a ``limit'' functor,
enabling to define it in other several cases.

Some properties of the properties of the Higson-Roe compactification
and the topological coarse structure are keeped, like to be
pseudoinverses. Furthermore, we describe some equivalences of
categories and other kind of funtorial relations.

In Section \ref{section2} we introduce the basic needed definitions
and notation in compactifications and algebras of functions, coarse
geometry and Z-sets in the Hilbert and the finite dimensional cube.

On Section \ref{section3}, mainly technical, we introduce the
`limit' and `total' operator, and characterize them in terms of
pseudometrics and algebras of functions, in order to obtain the main
results of the following section.

Section \ref{section4}, which is the core of the paper, contains the
results of this work:
\begin{itemize}
\item In Section \ref{section41} we give the definition of the generalized $C_0$
coarse structure attached to a family of pseudometrics. We prove
that it is equal to the topological coarse structure attached to a
compactification when we consider a family of pseudometrics which
define the topology of that compactification. Moreover, we define
the functor attached to the topological coarse structure (to do it,
we need to define the ``asymptotically continuous maps''). Also, we
characterize some coarse properties, like coarseness.

\item In section \ref{section42} we define the functor related with the Higson-Roe
compactification and study the needed coarse conditions to define
extensions of maps.

\item In section \ref{section43} we give an alternative topological definition of
the Higson-Roe functor as a limit functor, enabling to define it in
other several cases.

\item Finally, in Section \ref{section44} we put together all the functorial
information of the preceding two subsections and give some
equivalences of categories.\end{itemize}

\section{Preliminaries: Basic definitions and notations}\label{section2}

If $X$ is a set, $Y\subset X$ and $A$ is a family of functions over
$X$, pseudometrics over $X$ etc., we denote by $A|_Y$ the family
$\{f|_Y:f\in A\}$.

If $f:Z\rightarrow Z'$ is a (not necessarily continuous) map between
locally compact spaces, $f$ is proper if and only if for every
relatively compact subset $K\subset\hxp$, $f^{-1}(K)$ is relatively
compact (equivalently if for every met $\{x_\lambda\}\subset\hx$
with $x_\lambda\rightarrow\infty$, we have that
$f(x_\lambda)\rightarrow\infty$, see Proposition \ref{equiinfinv}).

If $(Z,\dd)$ is a metric space, $\dd$ is totally bounded if the
Cauchy completion is compact.

Let us give a brief summary of compactifications and algebras of
functions, theory of pseudometrics, coarse geometry and Z-sets in
the Hilbert cube and the finite dimensional cubes.

\textbf{Compactifications and algebras of functions.}

Let $\hx$ be a locally compact, but not compact, Hausdorff Space.
For us, a compactification of $\hx$ is a compact Hausdorff space $K$
containing $\hx$ as a dense subset, in which case, $\hx$ is open in
$K$. The corona of $K$ is $K\bs \hx$.

From \cite{jpmd} we take the following notation: we say that
\textit{compactification pack} is a vector $\xid$ such that $\tx$ is
a compact Hausdorff space, $X$ is a nowheredense closed subset of
$\tx$ and $\hx=\tx\bs X$. Observe that $\tx$ is a compactification
of $\hx$ and $X$ is its corona.

Given two compactifications $K$ and $K'$ of $\hx$, we say that
$K\leq K'$ if there exists a quotient $q:K'\rightarrow K$ such that
$q|_\hx=Id_\hx$ (equivalently, if $Id:\hx\rightarrow \hx$ extends to
a continuous map $q:K'\rightarrow K$). $K$ and $K'$ are equivalent
if $K\leq K'$ and $K'\leq K$, i.e., there exists a homeomorphism
$h:K\rightarrow K'$ such that $h|_\hx=Id_\hx$.

If $Z$ is a locally compact (maybe compact) Hausdorff space, we
denote by $C(Z)$ the collection of all the real continuous functions
$f:X\rightarrow \Rset$, by $C_0(Z)$ the collection of all the real
continuous functions which vanish at infinity and by $C_b(Z)$ the
collection of all the real continuous and bounded functions.

Given $A\subset C_b(\hx)$ containing $C_0(\hx)$, there is a natural
embedding $i_A:Z\hookrightarrow \Rset^{A}$, $x\rightarrow
(f(x))_{f\in A}$. In particular, $\overline{i_A(Z)}$ is a
compactification of $Z$.

In fact, there is a bijection between  the compactifications of
$\hx$ and the closed subalgebras of $C_b(\hx)$ containing $C_0(X)$,
given by:
\begin{itemize}
\item If $K$ is a compactification,
$C(K)|_\hx$ is an algebra satisfying that properties.
\item If $A$ an algebra as stated, $\overline{i_A(\hx)}$ is a
compactification.\end{itemize}

That bijection preserves the order, that is if $K\leq K'$, then
$C(K)|_\hx\subset C(K')|_\hx$. The smallest compactification is the
Alexandrov one, $\hx\cup\{\infty\}$, denoted here by $A$, attached
to $C_0(\hx)+\langle 1\rangle$, where $1$ is the constantly $1$
function. The biggest compactification of $\hx$ is the
Stone-\v{C}ech compactification, denoted usually by $\beta\hx$,
attached to $C_b(\hx)$. $\beta\hx$'s corona is often denoted by
$\hx^*$.

\textbf{Pseudometrics.}

A pseudometric over $X$ is a map $\dd:X\times X\rightarrow
[0,\infty)$ such that:
\begin{itemize}
\item $\dd(x,x)=0\qquad\forall x\in X$
\item $\dd(x,y)=\dd(y,x)\qquad\forall x,y\in X$
\item $\dd(x,z)\leq \dd(x,y)+\dd(y,z)\qquad x,y,z\in
X$\end{itemize}

If $X$ is a topological space, we say that $\dd$ is a pseudometric
of $X$ if it is continuous.

A pseudometric $\dd$ is a metric when $\dd(x,y)=0$ if and only if
$x=y$. If $x\in X$ y $r\geq0$, the $\dd$-ball, denoted by
$\BB_{\dd}(x,r)$ is the set $\{y\in X:\dd(x,y)<r\}$. The closed
$\dd$-ball, denoted by $\BBC_\dd(x,r)$ is the set $\{y\in
X:\dd(x,y)\leq r\}$

A family of pseudometrics $\ddd$ in a set generes a topology on
$\hx$, denoted by us by $\calt_\ddd$, given by the basis:
$$\{\BB_{\dd_1}(x,\epsil)\cap\cdots\cap
\BB_{\dd_n}(x,\epsil):x\in X,\epsil>0,n\in \Nset, d_1,\cdots,d_n\in
\ddd\}$$

In this topology, a net $\{x_\lambda\}\subset X$ converges to a
point $x$ if and only if $\dd(x_\lambda,x)\rightarrow 0$ for every
$\dd\in\ddd$. Moreover, $X$ is Hausdorff if and only if $\ddd$
separates points, i. e. for every $x,y\in X$ there exist
$\dd\in\ddd$ such that $\dd(x,y)>0$.

If $f:X\rightarrow \Rset$ is a continuous function, it induces a
pseudometric in $X$, denoted here by $\dd_f$, given by $\dd_f(x,y)=
|f(x)-f(y)|$.

If $X$ is Hausdorff, its topology is generated by a family of
pseudometrics if and only if $X$ is completely regular (see Theorem
10.6 of \cite{dg}, pág. 200). In this case, $\{\dd_f:f\in C(X)\}$ is
a family of pseudometrics generating its topology. In particular,
the topology of a locally compact Hausdorff space is generated by a
family of pseudometrics.

If $X$ is a compact Hausdorff space a family of pseudometrics $\ddd$
fo $X$ generes its topology if and only if it separates points
(indeed, $Id:(X,\calt)\rightarrow (X,\calt_\ddd)$ is a continuous
bijection between a compact and a Hausdorff space, hence is a
homeomorfism and $\calt=\calt_\ddd$).

\subsection{Coarse geometry}

Let us give some definitions of coarse geometry. For more
information, see \cite{cg}. Let $E,F\subset Z\times Z$, let $x\in Z$
and let $K\subset Z$. The product of $E$ and $F$, denoted by $E\circ
F$, is the set $\{(x,z):\exists y\in Z\textrm{ such that }(x,y)\in
E,(y,z)\in F\}$, the inverse of $E$, denoted by $E^{-1}$, is the set
$E^{-1}=\{(y,x):(x,y)\in E\}$, the diagonal, denoted by $\Delta$, is
the set $\{(z,z):z\in Z\}$. If $x\in Z$, the $E$-ball of $x$,
denoted by $E_x$ is the set $E_x=\{y:(y,x)\in E\}$ and, if $K\subset
Z$, $E(K)$ is the set $\{y:\exists x\in K\textrm{ such that }
(y,x)\in E\}$. If $\alpha$ is a family of subsets of $Z$,
$E(\alpha)$ is the family of subsets of $Z$ $\{E(U):U\in\alpha\}$.
We say that $E$ is symmetric if $E=E^{-1}$.

A coarse structure $\ce$ over a set $Z$ is a family of subsets of
$Z\times Z$ which contains the diagonal and is closed under the
formation of products, finite unions, inverses and subsets. The
elements of $\ce$ are called controlled sets. $B\subset Z$ is said
to be bounded if there exists $x\in Z$ and $E\in\ce$ with $B=E_x$
(equivalently, $B$ is bounded if $B\times B\in\ce$).

A map $f:(Z,\ce)\rightarrow (Z',\cep)$ between coarse spaces is
called bornologous if $f\times f(E)$ is controlled for every
controlled set $E$ of $Z$ and (coarsely) proper if $f^{-1}(B)$ is
bounded for every bounded subset $B$ of $Z'$. If $f$ is proper and
bornologous, it is said to be coarse.

 We say that $f$ is a coarse
equivalence if $f$ is coarse and there exists a coarse map
$g:(Z',\cep)\rightarrow (Z,\ce)$ such that $\{(g\circ f(x),x):x\in
Z\}\in\ce$ and $\{(f\circ g(y),y):y\in Z'\}\in \cep$. In this case,
$g$ is called a coarse inverse of $f$.

If $Z$ is a topological space and $E\subset Z\times Z$, we say that
$E$ is proper if $E(K)$ and $E^{-1}(K)$ are relatively compact for
every relatively compact subset $K\subset Z$. If $\ce$ is a coarse
structure over $Z$, we say that $(Z,\ce)$ is a proper coarse space
if $Z$ is Hausdorff, locally compact and paracompact, $\ce$ contains
a neighborhood of the diagonal in $Z\times Z$ and the bounded
subsets of $(Z,\ce)$ are precisely the relatively compact subsets of
$Z$.

Let us give the following definition: If $Z$ is a topological space
and $\ce$ is a coarse structure over $Z$, we say that that $(Z,\ce)$
is preproper if all its controlled subsets are proper and, for every
$K\subset Z$ and all the relatively compact subsets of $Z$ are
bounded in $\ce$.

Clearly, a proper coarse space is preproper. Observe that if
$(Z,\ce)$ is preproper, then any $B\subset Z$ is bounded if and only
if is locally compact (Indeed, if $B$ si bounded and not empty,
taking $x_0\in B$, we have that $B=B\times B(x_0)$ is relatively
compact because $\{x_0\}$ relatively compact and  $B\times B$ is
controlled and, consequently, proper).

Observe that if $f:(Z,\ce)\rightarrow(Z',\ce)$ are preproper coarse
spaces, then $f$ is coarsely proper if and only if $f$ is
topologically proper.

\subsection{Z-sets in the Hilbert cube or in a finite dimensional
cube}\label{introduccionzsets}

If $(\tx,\dd)$ is a compact metric space, $X\subset\tx$ is a Z-set
if for every $\epsil>0$ there exists a continuous map
$f:\tx\rightarrow\tx$ such that $\dd'(f,Id)<\epsil$
---where $\dd'$ is the supremum metric--- and $f(\tx)\cap
X=\varnothing$ (the definition of Z-set given in \cite{hc}, chapter
I-3, page 2, is trivially equivalent in this context).

By $Q$ we denote the Hilbert cube $[0,1]^\Nset$. $X$ is a Z-set of
$Q$ if and only if there exist an homeomorphism $h:[0,1]\times
Q\rightarrow Q$ such that $h(X)\subset\{0\}\times Q$ (see Remark 26
of \cite{jpmd}, pag. 106, for a proof). $X$ is a Z-set of the finite
dimensional cube $[0,1]^n$ (for $n\geq 1$) if and only if
$X\subset[-1,1]^n\bs (-1,1)^n$ (it follows from Example VI 2 of
\cite{hw}).

If $\tx$ is the Hilbert cube or a finite dimensional cube, $X\subset
\tx$ is a Z-set if and only if there exists an homotopy $H:\tx\times
[0,1]\rightarrow \tx$ such that $H_0=Id_Q$ and $H_t(\tx)\subset
\tx\bs X$ for every $t>1$. (Sufficiently is obvious, necessity
follows from characterizations above).

Every compact metric space has an embedding in $Q$ as a Z-set (it
has an embedding in $Q$, and hence en in $\{0\}\times Q\subset
[0,1]\times Q\approx Q$) and every compact subset of finite
dimension has an embedding as a Z-set in a finite dimensional cube
(if $\dim X=n$, it can be embebed in $[0,1]^{2n+1}$ (see \cite{hw},
Theorem V.2, pag. 56), so it can be embebed in a face of
$[0,1]^{2n+2})$.

\section{Previous technical results}\label{section3}

To describe the (extended) $C_0$ coarse structure and the functors
involved, we need to develop some tools which relate
compactification packs with pseudometrics and algebras of functions.

The reason is that the classical theory of ring of functions (see
\cite{rcf}) is not enough to our purpose, because we work with not
necessarily continuous functions (but with topological properties,
like properness).

\subsection{Limits at infinity, properness}

\begin{lema}[technical]\label{epropiolimitesprevio} Let $\hx$ and
$\hxp$ be locally compact spaces and $E\subset\hx\times\hxp$. For
every $K\subset\hxp$, let us denote by $E(K)$ the set $\{x\in
\hx:\exists y\in K\textrm{ such that }(x,y)\in E\}$. Then,
\begin{enumerate}[a) ]
\item $E(K)$ is relatively compact for every relatively compact set $K\subset\hxp$.
\item For every net $(x_\lambda,y_\lambda)\subset E$, if
$x_\lambda\rightarrow\infty$, then $y_\lambda\rightarrow\infty$.
\end{enumerate}
are equivalent.
\end{lema}

\begin{proof}

Suppose a). Pick $\{(x_\lambda,y_\lambda)\}\subset E$ with
$x_\lambda\rightarrow\infty$. Take $K\subset \hxp$ relatively
compact. Then $E(K)$ is relatively compact, so there exists
$\lambda_0$ such that for every $\lambda\geq\lambda_0$, $x_\lambda
\not\in E(K)$. Then, $y_\lambda\not\in K$ for all $\lambda\geq
\lambda_0$. Therefore, $y_\lambda\rightarrow\infty$.

Suppose b). Pick $K\subset \hxp$ relatively compact. Suppose $E(K)$
is not relatively compact. Then, for all $B\subset \hx$ relatively
compact, there exists $x_B\in E(K)\bs B$. For all $B$, let $y_B\in
K$ such that $(x_B,y_B)\in E$. Let $D$ be the directed set
consisting of all the relatively compact sets of $\hx$ with the
order defined by $B\leq B'$ if and only if $B\subset B'$. Then
$\{x_B\}_{B\in D}$ is a net such that $x_B\rightarrow\infty$, since
$x_B\not\in B'$ for every $B'\geq B$. Hence, $y_B\rightarrow\infty$
which contradicts the fact that $\{y_B\}\subset K$, which is
relatively compact. It follows that $E(K)$ is relatively compact.
\end{proof}

\begin{prop} \label{epropiolimites} If $\hx$ is locally compact and $E\subset \hx\times \hx$, then
\begin{enumerate}[a) ]
\item $E$ is proper.
\item For every net $\{(x_\lambda,y_\lambda)\}\subset E$, we have that
$x_\lambda\rightarrow\infty$ if and only if
$y_\lambda\rightarrow\infty$.\end{enumerate} are equivalent.

\end{prop}
\begin{proof}
Apply Lemma \ref{epropiolimitesprevio} to $E$ and $E^{-1}$.\end{proof}

\begin{cor}\label{epropiolimites2}
If $\hx$ is locally compact and $E\subset \hx\times \hx$, then
\begin{enumerate}[a) ]
\item $E$ is proper.
\item  For every net $\{(x_\lambda,y_\lambda)\}\subset E\cup
E^{-1}$, if $x_\lambda\rightarrow\infty$, then
$y_\lambda\rightarrow\infty$.
\item For every net $\{(x_\lambda,y_\lambda)\}\subset E\cup
E^{-1}$, if $x_\lambda\not\rightarrow\infty$, then
$y_\lambda\not\rightarrow\infty$. \end{enumerate} are equivalent.
\end{cor}

\begin{rmk}\label{obsepropioinfty} If $\hx$ is locally compact and
$\{(x_\lambda,y_\lambda)\}\subset\hx\times\hx$, then
$(x_\lambda,y_\lambda)\rightarrow\infty$ if and only if
$x_\lambda\rightarrow\infty$ or $y_\lambda\rightarrow\infty$. If,
moreover, $E\subset\hx\times\hx$ is proper and
$\{(x_\lambda,y_\lambda)\}\subset E$, then
$(x_\lambda,y_\lambda)\rightarrow\infty$ if and only if
$x_\lambda\rightarrow\infty$ and
$y_\lambda\rightarrow\infty$.\end{rmk}

\begin{lema} \label{equiinfinv}
Let $f:\hx\rightarrow \hxp$ be a map between locally compact spaces. Then,
\begin{enumerate}[a) ]
\item $f^{-1}(K')$ is relatively compact for every relatively compact subset $K'$ of $\hxp$.
\item For every net $\{x_\lambda\}\subset\hx$, if $x_\lambda\rightarrow\infty$, then $f(x_\lambda)\rightarrow\infty$.
\end{enumerate}
are equivalent. Moreover
\begin{enumerate}[a') ]
\item $f(K)$ is relatively compact for every relatively compact subset $K$ of $\hx$.
\item For every net $\{x_\lambda\}\subset \hx$, if $f(x_\lambda)\rightarrow\infty$, then
$x_\lambda\rightarrow\infty$.\end{enumerate} are equivalent. And
\begin{enumerate}[a'') ]
\item For every $K\subset\hx$, $f(K)$ is relatively compact if and only if $K$ is relatively compact.
\item For every net $\{x_\lambda\}\subset \hx$, $x_\lambda\rightarrow\infty$
if and only if $f(x_\lambda)\rightarrow\infty$.\end{enumerate} are
equivalent.
\end{lema}

\begin{proof}
To see the first and the second equivalences, apply Lemma
\ref{epropiolimitesprevio} to $\{(x,f(x)):x\in \hx\}$ and
$\{(f(x),x):x\in \hx\}$ respectively. The last equivalence is a
consequence of the other ones.
\end{proof}

\begin{deff} Let $f: \hx\rightarrow \hy$ be a map between locally compact spaces. $f$ is biproper if, for every
$K\subset \hx$, $K$ is relatively compact if and only $f(K)$ is
(equivalently if, for every net $\{x_\lambda\}\subset \hx$,
$x_\lambda\rightarrow\infty$ if and only if
$f(x_\lambda)\rightarrow\infty$).
\end{deff}

\begin{exa}\label{ejsbipropiaa}
If $f:\hx\rightarrow \hy$ is a continuous and proper map between
locally compact spaces, then $f$ is biproper. Indeed, if $f$ is
continuous and $K$ relatively compact, then $f(K)$ is relatively
compact.
\end{exa}

\begin{exa}\label{exabiproper}
If $f:\hx\rightarrow \hy$ is a coarse map between preproper spaces,
then $f$ is biproper. (Indeed, if $K$ is a relatively compact subset
of $\hx$, then $K\times K$ es controlled, so that $f(K)\times
f(K)=f\times f(K\times K)$ is controlled. Hence, $f(K)$ is bounded
and, consequently, relatively compact).
\end{exa}

\subsection{Limit and total operators $l$ and $t$, for maps}

\begin{prop}\label{extensioncontinua}
Let $\tx$ be a topological space, $\hx$ a dense subset of $\tx$ and
$A\subset \hx$. Suppose $Y$ is a regular space and $f:\tx\rightarrow
Y$ is a map such that
\begin{equation}\label{eq15a9573jkla}
\forall x\in A\,\,\,f(x)=\lim_{\substack{z\rightarrow
x\\z\in\hx}}f(z)
\end{equation}
Then, $f|_A$ is continuous.\end{prop}

\begin{proof}

Take $x\in A$ and $\{x_\lambda\}_{\lambda\in\Lambda}\subset A$ such
that $x_\lambda\rightarrow x$. To prove the continuity of $f|_A$,
let us show that $f(x_\lambda)\rightarrow f(x)$. Let $V$ be a
neighborhood of $f(x)$ and let $W$ be an open neighborhood of $f(x)$
such that $\overline W\subset V$.

Since $\hx$ is dense, for all $u\in\Lambda$ there is a net
$\{z_{\sigma^u}\}_{\sigma^u\in \Sigma^u}\subset\hx$ such that
$z_{\sigma^u}\rightarrow x_u$. Thus, by (\ref{eq15a9573jkla}),
$f\big(z_{\sigma^u}\big)\rightarrow f(x_u)$ for all $u\in\Lambda$.

Consider the cofinal ordered set
$\Lambda\times\prod_{u\in\Lambda}\Sigma^u$ with the order
$\big(\lambda_0,(\sigma_0^u)_{u\in\Lambda}\big)\leq
\big(\lambda_1,(\sigma^u_1)_{u\in\Lambda}\big)$ if
$\lambda_0\leq\lambda_1$ and $\sigma_0^u\leq\sigma^u_1$ for every
$u\in\Lambda$.

Put
$z_{\big(\lambda,(\sigma^u)_{u\in\Lambda}\big)}=z_{\sigma^\lambda}$
for every
$\big(\lambda,(\sigma^u)_{u\in\Lambda}\big)\in\Lambda\times\prod_{u\in\Lambda}\Sigma^u$.

Let us see that
\begin{equation}\label{eq15a9573jklauu}
z_{\big(\lambda,(\sigma^u)_{u\in\Lambda}\big)}\rightarrow x
\end{equation}

Suppose $U$ is an open neighborhood of $x$ in $\tx$. Let us define
$\big(\lambda_0,(\sigma_0^u)_{u\in\Lambda}\big)$ as follows: choose
$\lambda_0$ such that $x_{\lambda}\in U$ for every
$\lambda\geq\lambda_0$. For all $u\geq\lambda_0$, since $x_u\in U$
and $z_{\sigma^u}\rightarrow x_u$, we may choose $\sigma_0^u$ with
$z_{\sigma^u}\in U$ for every $x_{\sigma^u}\geq x_{\sigma_0^u}$. For
all $u\not\leq\lambda_0$, take any $\sigma_0^u\in \Sigma^u$.

Hence, for every $\big(\lambda,(\sigma^u)_{u\in\Lambda}\big)\geq
\big(\lambda_0,(\sigma_0^u)_{u\in\Lambda}\big)$, we have that
$\lambda\geq\lambda_0$ and $\sigma^u\geq\sigma_0^{\lambda}$. Then,
$z_{\big(\lambda,(\sigma^u)_{u\in\Lambda}\big)}=z_{\sigma^\lambda}\in
U$.

(\ref{eq15a9573jklauu}) and (\ref{eq15a9573jkla}) show that
$f\Big(z_{\big(\lambda,(\sigma^u)_{u\in\Lambda}\big)}\Big)\rightarrow
f(x)$. Therefore, there exists
$\big(\lambda_1,(\sigma_1^u)_{u\in\Lambda}\big)$ such that for all
$\big(\lambda,(\sigma^u)_{u\in\Lambda}\big)\geq
\big(\lambda_1,(\sigma_1^u)_{u\in\Lambda}\big)$ we have that
$f\Big(z_{\big(\lambda,(\sigma^u)_{u\in\Lambda}\big)}\Big)\in W$.

Fix $\lambda\geq\lambda_1$ and
$\sigma_2^\lambda\geq\sigma_1^{\lambda}$. Take, for every
$u\in\Lambda$, $\sigma^u=\sigma_2^\lambda$ if $u=\lambda$ and
$\sigma^u=\sigma_1^u$ if $u\neq\lambda$. Hence,
$\big(\lambda,(\sigma^u)_{u\in\Lambda}\big)\geq
\big(\lambda_1,(\sigma_1^u)_{u\in\Lambda}\big)$, Thus:
$$f\big(z_{\sigma^\lambda_2}\big)=
f\Big(z_{\big(\lambda,(\sigma^u)_{u\in\Lambda}\big)}\Big)\in W$$

Therefore, $f\big(z_{\sigma_2^\lambda}\big)\in W$ for every
$\sigma_2^\lambda\geq\sigma_1^{\lambda}$, so
$f(x_{\lambda})=\lim_{\sigma^\lambda}
f\big(z_{\sigma^\lambda}\big)\in \overline W\subset V$. Hence,
$f(x_{\lambda})\in V$ for every $\lambda\geq\lambda_1$ and
$f(x_{\lambda})\rightarrow f(x)$.\end{proof}

\begin{deff}[Limit and total operators $l$ and $t$]
Let $\xiv$ be a compactification pack, $Y$ a topological space and
$\hf:\hx\rightarrow Y$ a map. If it can be defined, the limit
function $l(\hf):X\rightarrow Y$ is the one such that for every
$x\in X$
$$l(\hf)(x)=\lim_{\substack{z\rightarrow x\\z\in\hx}}\hf(z)$$.

In this case, the total function $t(\hf):\tx\rightarrow Y$ is the
one such that for every $x\in\tx$
$$t(\hf)(x)=\left\{\begin{array}{ll}\hf(x)&\text{if }x\in\hx\\
l(\hf)(x)&\text{if }x\in X\end{array}\right.$$
\end{deff}

\begin{prop}\label{lhfescont} Let $\xiv$ be a compactification pack and $Y$ a completely regular space.
Suppose $\hf:\hx\rightarrow Y$ is a map such that $l(f)$ is defined. Then:
\begin{enumerate}[a) ]
\item $l(f)$ is continuous.
\item If $f$ is continuous, then $t(f)$ is continuous.
\end{enumerate}\end{prop}

\begin{proof} a) and b) follow from Proposition \ref{extensioncontinua},
by taking $A=X$ and $A=\tx$, respectively.
\end{proof}

\begin{rmk} If $\xiv$ is a compactification pack and
$ f:\hx\rightarrow\Rset$ is a bounded function, then $ f$ we can be
seen as a function $ f:\hx\rightarrow\BBC\left(0,\|
f\|_\infty\right)$, being $\BBC\left(0,\| f\|_\infty\right)$
compact.
\end{rmk}

\begin{lema}\label{lemaparasumaretc} Let $\xiv$ be a compactification pack and $Y$ a topological space. Consider a
sequence $\{Y_i\}_{i=1}^n$ of topological spaces and one $\{
f_i:\hx\rightarrow Y_i\}_{i=1}^n$ of maps such that $l( f_i)$ is
defined for each $i$. Let $g:Y_1\times\cdots\times Y_m\rightarrow Y$
be a continuous function. Then, the map $\phi:\hx\rightarrow Y$,
$x\rightarrow g\big( f_1(x),\cdots, f_n(x)\big)$ satisfies
$l(\phi)(x)=g\big(l( f_1)(x),\cdots,l( f_n)(x)\big)$ for every $x\in
X$.\end{lema}

\begin{proof}
For every $x\in X$:
$$l(\phi)(x)= \lim_{\substack{z\rightarrow
x\\z\in\hx}}g\big( f_1(z),\cdots, f_n(z)\big)=$$
$$g\big(\lim_{\substack{z\rightarrow x\\z\in\hx}}
f_1(z),\cdots,\lim_{\substack{z\rightarrow x\\z\in\hx}} f_n(z)\big)=
g\big(l( f_1)(x),\cdots,l( f_n)(x)\big)$$\end{proof}

From Lemma \ref{lemaparasumaretc} we get:

\begin{exa}\label{ejlsuma}
If $\xiv$ is a compactification pack and $f_1,f_2:\hx\rightarrow
\Rset$ are maps such that $l(f_1)$ y $l(f_2)$ are defined, then
$l(f_1+f_2)=l(f_1)+f(l_2)$ and $l(f_1\cdot f_2)=l(f_1)\cdot l(f_2)$.

\end{exa}
\begin{exa}
If $\xiv$ is a compactification pack, $Y$ is regular,
$f:\hx\rightarrow Y$ is such that $l( f)$ is defined and $g\in
C(Y)$, then $l(g\circ f)=g|_X\circ l( f)$.
\end{exa}
\begin{exa}
If $\xiv$ is a compactification pack, $Y$ is regular, $ f_1,
f_2:\hx\rightarrow Y$ are such that $l( f_1)$ and $l( f_2)$ are
defined and $\dd'$ is a pseudometric of $Y$, then the funcion
 $\phi:\hx\rightarrow \Rset$, $\phi(z)=\dd( f_1(z), f_2(z))$ satisfies
$l(\phi)(x)=\dd'\big(l( f_1)(x),l( f_2)(x)\big)$ for every $x\in
X$.\end{exa}

From Example \ref{ejlsuma}, we get:

\begin{lema} If $\xiv$ is a compactification pack, then
$$B_l(\hx,\tx)=\{\hx:\hx\rightarrow \Rset\text{ bounded such that }l( f)\text{ is defined}\}$$
is an algebra and $l:B_l(\hx,\tx)\rightarrow C(X)$, a morphism of
algebras.

\end{lema}

\begin{deff} If $\hx$ is a locally compact space, then $B_0(\hx)$ is the family of bounded functions
$f:\hx\rightarrow\Rset$ vanishing at infinity (i. e., such that
$\lim_{z\rightarrow\infty}f(z)=0$).\end{deff}

\begin{lema}\label{caractb0limit} If $\xiv$  is a compactification pack and $f:\hx\rightarrow \Rset$ a bounded map, then
$f\in B_0(\hx)$ if and only if $l( f)=0$.\end{lema}

\begin{proof}If $ f\in B_0(\hx)$, $x\in X$ and $
\{z_\lambda\}\subset\hx$ with $z_\lambda\rightarrow x$, then
$z_\lambda\rightarrow\infty$ in $\hx$ and, hence $l( f)(x)=\lim
f(z_\lambda)=0$. Thus, $l(f)=0$.

Suppose now that $l( f)=0$. Fix $\epsil>0$. Since for every $x\in X$
$\lim_{\substack{z\rightarrow x\\z\in\hx}} f(z)=l( f)(x)=0$, there
exists a open neighborhood $U_x$ of $x$ in $\tx$ such that $|
f(z))|<\epsil$ in $U_x\cap \hx$. Consider the compact subset of
$\hx$ $K=\tx\bs \bigcup_{x\in X}^n U_{x}$. Clearly, $|f(x)|<\epsil$
outside $K$ and hence, $f\in B_0(\hx)$.
\end{proof}

\begin{lema}\label{algebrabl} If $\xiv$ is a compactification pack, then
$B_l(\hx,\tx)=C(\tx)|_\hx+B_0(\hx)$.\end{lema}
\begin{proof}
Clearly, $C(\tx)|_\hx,B_0(\hx)\subset B_l(\hx,\tx)$, then we have
one inclusion. To see the other, take $f\in B_l(\hx,\tx)$. Since
$L(f):X\rightarrow \Rset$ is continuous, there is a continuous
extension $g:\tx\rightarrow \Rset$. Then, $g|_\hx\in C(\tx)|_\hx$.
Observe that $L(f-g|_\hx)=L(f)-L(g|_\hx)=L(f)-g|_X=L(f)-L(f)=0$.
Thus $f-g|_\hx\in B_0(\hx)$. Therefore, $f=g|_\hx+(f-g|_\hx)\in
C(\tx)|_\hx+B_0(\hx)$.\end{proof}

\begin{lema}\label{xinftysiiyinfty} Let $\xiv$ be a compactification
pack.
Suppose $\ddd$ is a family of pseudometrics which genere $\tx$'s
topology and let $\{(x_\lambda,y_\lambda)\}\subset\hx\times\hx$ be a
net such that $\dd(x_\lambda,y_\lambda)\rightarrow 0$ for all
$\dd\in\ddd$. Then:
\begin{enumerate}[a) ]
\item $x_\lambda\rightarrow\infty$ in $\hx$ if and only if
$y_\lambda\rightarrow\infty$ in $\hx$
\item $(x_\lambda,y_\lambda)\rightarrow\infty$ in $\hx\times\hx$ if and only if
$x_\lambda\rightarrow\infty$ and $y_\lambda\rightarrow\infty$ in
$\hx$.
\end{enumerate}\end{lema}

\begin{proof}
If $x_\lambda\not\rightarrow\infty$, then there exists a convergent
subnet $x_{\lambda'}\rightarrow x\in\hx$. Since, for every $\dd\in
\ddd$,
$\dd(y_{\lambda'},x)\leq\dd(y_{\lambda'},x_{\lambda'})+\dd(x_{\lambda'},x)\rightarrow
0$, we get $y_{\lambda'}\rightarrow x$ and hence
$y_\lambda\not\rightarrow\infty$. By symmetry,
$y_\lambda\not\rightarrow\infty$ implies
$x_\lambda\not\rightarrow\infty$ and we get a). b) is easily deduced
from a) taking into account that
$(x_\lambda,y_\lambda)\rightarrow\infty$ if and only if
$x_\lambda\rightarrow\infty$ or
$y_\lambda\rightarrow\infty$.\end{proof}

\begin{lema}[technical]\label{caractconvcompto} Let $X$ be a compact Hausdorff
space, $x\in X$ and $\{x_\lambda\}\subset X$ a net. Then
$x_\lambda\rightarrow x$ if and only if for all convergent subnet
$x_{\lambda'}\rightarrow y\in X$ we have that $y=x$.\end{lema}

\begin{proof} Since $X$ is Hausdorff, necessity is obvious. To prove sufficiency, suppose  $x_\lambda\not\rightarrow x$. Then, there exists a open neighborhood
 $U$ of $x$ and a subnet $x_{\lambda'}$ such that
$x_{\lambda'}\not\in U$ for all $\lambda'$. Since $X$ is compact,
$x_{\lambda'}$ has a convergent subnet $x_{\lambda''}\rightarrow y$.
$y\in X\bs U$, because $X\bs U$ is closed, and hence $y\neq
x$.\end{proof}

\begin{lema}\label{lemaequil1}Let $\xiv$ be a compactification pack and $Y$ a compact Hausdorff space.
Suppose $\ddd$ and $\ddd'$ are families of pseudometrics of $\tx$
and $Y$ which generate their topologies respectively. Consider a map
$f:\hx\rightarrow Y$. Then,
\begin{enumerate}[a) ]
\item $l( f)$ is defined.
\item For every net $\{(x_\lambda,y_\lambda)\}\subset\hx\times\hx$:
if $(x_\lambda,y_\lambda)\rightarrow\infty$ and
$\dd(x_\lambda,y_\lambda)\rightarrow 0$ for every $\dd\in\ddd$, then
 $\dd'( f(x_\lambda), f(y_\lambda))\rightarrow 0$ for every
$\dd'\in\ddd'$.\end{enumerate} are equivalent.\end{lema}

\begin{proof}
Suppose b). Fix $x\in X$ and take a net $\{x_\lambda\}_{\lambda\in
\Lambda}\subset\hx$ with $x_\lambda\rightarrow x$. Since
$\{f(x_\lambda)\}\subset Y$ and $Y$ is compact, there exists a
subnet $\{x_{\lambda'}\}_{\lambda'\in \Lambda'}$ such that $
f(x_{\lambda'})\rightarrow x'\in Y$.

Consider a net $\{y_\sigma\}_{\sigma\in\Sigma}\subset\hx$ with
$y_\sigma\rightarrow x$. Let us show that $f(y_\sigma)\rightarrow
x'$ by using Lemma \ref{caractconvcompto}'s characterization.

Let $\{y_{\sigma'}\}_{\sigma'\in\Sigma'}$ be any subnet with
$f(y_{\sigma'})\rightarrow z'$. Consider the net
$\{(x_{(\lambda',\sigma')},y_{(\lambda',\sigma')})\}_{(\lambda',\sigma')\in\Lambda'\times\Sigma'}$,
where $x_{(\lambda',\sigma')}=x_{\lambda'}$ and
$y_{(\lambda',\sigma')}=y_{\sigma'}$. Clearly,
$x_{(\lambda',\sigma')}\rightarrow x$,
$y_{(\lambda',\sigma')}\rightarrow x$ $
f\big(x_{(\lambda',\sigma')}\big)\rightarrow x'$ and $
f\big(y_{(\lambda',\sigma')}\big)\rightarrow z'$.

Since,
$\dd(x_{(\lambda',\sigma')},y_{(\lambda',\sigma')})\rightarrow
\dd(x,x)=0$ for every $\dd\in\ddd$, b) shows that $\dd'(x',z')=\lim
\dd'\Big( f(x_{(\lambda',\sigma')}), f(y_{(\lambda',\sigma')})\Big)=
0$, for every $\dd'\in\ddd'$. Thus, $z'=x'$ and $
f(y_\sigma')\rightarrow x'$. Therefore,
$\lim_{\substack{y\rightarrow x\\ y\in\hx}} f(y)= x'$ and $l( f)$ is
defined.

Suppose that b) doesn't hold. Let
$(x_\lambda,y_\lambda)\subset\hx\times\hx$ be such that
$x_\lambda\rightarrow\infty$, $\dd(x_\lambda,y_\lambda)\rightarrow
0$ for every $\dd\in\ddd$ but there exist $\dd'_0\in\ddd$,
$\epsil>0$ and a subnet $\{x_{\lambda'}\}$ such that $\dd'_0(
f(x_{\lambda'}), f(y_{\lambda'}))\geq\epsil$ for every $\lambda'$.
Since $\tx\times\tx\times Y\times Y $ is compact, there exists a
subnet $\lambda''$ of $\lambda'$ such that
$(x_{\lambda''},y_{\lambda''}, f(x_{\lambda''}),
f(y_{\lambda''}))\rightarrow (x,y,z,t)$.

$x=y$, because $\dd(x,y)=\lim\dd(x_{\lambda''},y_{\lambda''})=0$ for
every $\dd\in\ddd$. But $z\neq t$, because $\dd_0(z,t)=\lim(
f(x_{\lambda''}), f(y_{\lambda''}))\geq\epsil$. Therefore,
$x_{\lambda''}\rightarrow x$, $y_{\lambda''}\rightarrow x$ but $\lim
f(x_{\lambda''})\neq\lim f(y_{\lambda''})$, then
$\lim_{\substack{\widehat x\rightarrow x\\\widehat x\in\hx}}
f(\widehat x)$ is not defined and, consequently, neither $l( f)$.
\end{proof}

\begin{lema}\label{lemaequil2} Let $\xiv$  be a compactification pack and $Y$, a compact Hausdorff space. Suppose $\ddd'$
is a family of pseudometrics of $Y$ which genere its topology.
Consider the maps $f,g:\hx\rightarrow Y$, with $l(f)$ defined. Then,
\begin{enumerate}[a) ]
\item $l(g)=l(f)$
\item For every $\dd'\in \ddd'$,
$\lim_{z\rightarrow\infty}\dd'(f(z),g(z))=0$.
\end{enumerate}
are equivalent
\end{lema}
\begin{proof}  For every $\dd'\in\ddd'$, consider the function $\phi_{\dd'}:\hx\rightarrow\Rset$,
$z\rightarrow\dd'(f(z),g(z))$. Observe that b) is equivalent to say
that, for every $\dd'\in\ddd'$, $\phi_{\dd'}\in B_0(\hx)$, i. e.,
that $l(\phi_{\dd'})=0$.

If $l(g)=l(f)$, then for every $\dd'\in\ddd'$ and every $x\in X$,
$l(\phi_{\dd'})(x)=\dd'(l(f)(x),l(g)(x))=0$. Hence,
$l(\phi_{\dd'})=0$ and we have b).

Suppose b). Let us see first that $l(g)$ is defined, by using Lemma
\ref{lemaequil1}'s characterization. Consider a family of
pseudometrics $\ddd$ which genere $\tx$'s topology. Take
$\{(x_\lambda,y_\lambda)\}\subset\hx\times\hx$ such that
$(x_\lambda,y_\lambda)\rightarrow\infty$ and
$\dd(x_\lambda,y_\lambda)\rightarrow 0$ for every  $\dd\in\ddd$.
Thus, for every $\dd'\in\ddd'$, $\dd'(g(x_\lambda),g(y_\lambda))\leq
\dd'(g(x_\lambda),f(x_\lambda))+\dd'(f(x_\lambda),f(y_\lambda))+\dd'(f(y_\lambda),g(y_\lambda))\rightarrow
0$ and $l(g)$ is defined.

Pick $x\in X$. Since, for every $\dd'\in\ddd'$,
$\dd(l(f)(x),l(g)(x))=l(\phi_{\dd'})(x)=0$, we get $l(g)(x)=l(f)(x)$
and $l(g)=l(f)$.\end{proof}

\subsection{Limit and total operators $L$ and $T$, for maps}

We have defined $l$ and $t$ with the aim of working with proper
functions $\hf:\hx\rightarrow\hxp$, where $\xiv$ and $\xpiv$ are
compactification packs. But there, we need to do a little change in
the definition.

\begin{lema}\label{infinitopropia}Let $\xiv$ and $\xpiv$  be compactification packs and
$\hf:\hx\rightarrow\txp$ a map such that $\hf(\hx)\subset\hxp$ and
$l(\hf)$ is defined. Then, $l(\hf)(X)\subset X'$ if and only if
$\hf:\hx\rightarrow\hxp$ is proper.\end{lema}

\begin{proof}Suppose $l(\hf)(X)\not\subset X'$. Take $x\in X$ with
$l(\hf)(x)\in\hxp$ and $\{z_\lambda\}\subset \hx$ with
$z_\lambda\rightarrow x$. Since $\hf(z_\lambda)\rightarrow
l(\hf)(x)\in\hx$, we have that $z_\lambda\rightarrow\infty$ in
$\hx$, but $\hf(z_\lambda)\not\rightarrow\infty$ in $\hxp$ and we
get that $\hf:\hx\rightarrow\hxp$ is not proper.

Suppose now that $\hf:\hx\rightarrow\hxp$ is not proper. Take
$x_\lambda\rightarrow\infty$ such that
$\hf(x_\lambda)\not\rightarrow \infty$. Then, there exist a compact
subset $K'$ of $\hxp$ and a subnet $\{\hf(x_{\lambda'})\}\subset
K'$. By $\tx$'s compacity, we may take a subnet $x_{\lambda''}$ of
$x_{\lambda'}$ with $x_{\lambda''}\rightarrow x_0\in \tx$. But
$x_0\in X$, because $x_{\lambda''}\rightarrow\infty$ in $\hx$. Then,
$l(\hf)(x_0)=\lim\hf(x_\lambda)\subset K\subset\hx$ and we get
$l(\hf)(X)\not\subset X'$.
\end{proof}

Then, the following definition makes sense:

\begin{deff}[Limit and total operators $L$ and $T$] Let $\xiv$ and $\xpiv$ be compactification packs. Suppose
$\hf:\hx\rightarrow\hxp$ is a proper map (not necessarily
continuous). If it can be defined, the limit function
$L_{\tx\txp}(\hf):X\rightarrow X'$ is the one such that for every
$x\in X$:
$$L_{\tx\txp}(\hf)=\lim_{\substack{z\rightarrow x
\\z\in\hx }}\hf( z)$$.

In this case, the total function
$T_{\tx\txp}(\hf):\tx\rightarrow\txp$ is the one such that for every
$x\in  \tx$:
$$T_{\tx\txp}(\hf)(x)=\left\{\begin{array}{ll}\hf(x)&\textrm{ if $x\in \hx$}\\ L_{\tx\txp}(\hf)(x)&\textrm{ if $x\in
X$}\end{array}\right.$$

when no confusion arise, we denote $L_{\tx\txp}$ and $T_{\tx\txp}$
by $L$ and $T$, respectively.

\end{deff}

\begin{prop}  Let $\hf:\hx\rightarrow\hxp$ be a proper map, where  $\xiv$ and $\xpiv$ are compactification packs.
If $L(\hf)$ is defined, then it is continuous. If, moreover $\hf$ is
continuous, then $T(\hf)$ is.\end{prop}

\begin{proof}If follows from Proposition \ref{lhfescont}.\end{proof}

\begin{prop}\label{prefunctor} Let $\xiv$ and $\xpiv$ be compactification packs and $Y$ a topological space.
Consider a proper map $\hf:\hx\rightarrow\hxp$ a map
$g:\hxp\rightarrow Y$ such that $L(\hf)$ and $l( g)$ are defined.
Then, $l( g\circ\hf)=l( g)\circ L(\hf)$ and $t( g\circ\hf)=t(
g)\circ T(\hf)$.\end{prop}
\begin{proof} Pick $x\in X$ and $\{x_\lambda\}\subset\hx$ such that
$x_\lambda\rightarrow x$. Then, $\hf(x_\lambda)\rightarrow
L(\hf)(x)$ with $\{\hf(x_\lambda)\}\subset\hxp$ and $L(\hf)(x)\in
X'$. Thus, $ g(\hf(x_\lambda))\rightarrow l( g)\big(L(\hf)(x)\big)$.
Hence, $l( g\circ\hf)=l( g)\circ L(\hf)$ and, consequently, $t(
g\circ\hf)=t( g)\circ T(\hf)$.
\end{proof}

\begin{prop}\label{ltfunctoriales} Operators $L$ and $T$ are functors. That is, if $\xiv$, $\xpiv$
and $\xppiv$ are compactification packs and $\hf:\hx\rightarrow\hxp$
$\hg:\hxp\rightarrow\hxpp$ are maps with $L(\hf)$ and $L(\hg)$
defined, then:
\begin{enumerate}
\item $L(\hf\circ\hg)=L(\hf)\circ L(\hg)$ and $T(\hf\circ\hg)=T(\hf)\circ T(\hg)$.
\item $L(Id_{\hx})=Id_X$ and $T(Id_{\hx})=Id_\tx$
\end{enumerate}\end{prop}

\begin{proof}a) follows from
\ref{prefunctor} and b) is obvious.\end{proof}

\begin{rmk} Let $\xiv$ and $\xpiv$ be compactification packs. Suppose $K$
and $K'$ are compactifications of $\hx$ and $\hxp$ equivalent to
$\tx$ and $\txp$ respectively. Consider a proper map
$\hf:\hx\rightarrow\hxp$. Then, $L_{\tx,\txp}(\hf)$ is defined if
and only if $L_{K,K'}(\hf)$ is defined.

It follows from the fact that, if $h:\tx\rightarrow K$ and
$h':\txp\rightarrow K'$ are homeomorphisms such that $h|_\hx=Id_\hx$
and $h'|_\hxp=Id_\hxp$, then
$T_{K,K'}(\hf)=T_{K,K'}(Id_\hxp\circ\hf\circ Id_\hx)=
T_{\txp,K'}(Id_\hxp)\circ T_{\tx,\txp}(\hf)\circ T_{K,\tx}(\hf)=
h'\circ T_{\tx,\txp}(\hf)\circ h^{-1}$ .
\end{rmk}

\begin{lema}\label{caractdel}Let $\hf:\hx\rightarrow\hxp$ be a proper map, where $\xiv$ and $\xpiv$ are compactification packs .
Suppose $\ddd$ and $\ddd'$ are two families of pseudometrics which
genere the topologies of $\tx$ and $\txp$, respectively. Then,
\begin{enumerate}[a) ]
\item $L(\hf)$ is defined.

\item For every net $\{(x_\lambda,y_\lambda)\}\subset\hx\times\hx$:
if $(x_\lambda,y_\lambda)\rightarrow\infty$ and
$\dd(x_\lambda,y_\lambda)\rightarrow 0$ for every $\dd\in\ddd$, then
 $\dd'( f(x_\lambda), f(y_\lambda))\rightarrow 0$ for every
$\dd'\in\ddd'$.\

\item $\hf^*\big(C(\txp)|_\hxp\big)\subset =B_l(\hx,tx)=C(\tx)|_\hx+B_0(\hx)$.
\end{enumerate}
are equivalent.

\end{lema}

\begin{proof}
The equivalence between a) and b) follows from Lemma
\ref{lemaequil1}.

Let us see that a) implies c). Pick $h\in C(\txp)$. Since
$h|_{X'}\circ L(\hf):X\rightarrow\Rset$ is continuous, Tietze
extension theorem shows that there is a continuous extension
$g:\tx\rightarrow \Rset$. Hence,
 $l(h|_\hxp\circ\hf-g|_\hx)=l(h|_\hxp)\circ
L(\hf)-l(g|_\hx)=h|_{X'}\circ L(\hf)-g|_X=0$, thus
$h|_\hxp\circ\hf-g|_\hx\in B_0(\hx)$. Therefore,
$\hf^*(h|_\hxp)=h|_\hxp\circ\hf=g|_\hx+\big(h|_\hxp\circ\hf-g|_\hx\big)\in
C(\tx)|_\hx + B_0(\hx)$.

Let us see that c) implies b). Consider on $\tx$ and $\txp$ the
families of pseudometrics $\{\dd_g:g\in C(\tx)\}$ and $\{\dd_h:h\in
C(\txp)\}$ respectively. Choose a net
$\{(x_\lambda,y_\lambda)\}\subset\hx\times\hx$ such that
$(x_\lambda,y_\lambda)\rightarrow\infty$ and
$\dd_g(x_\lambda,y_\lambda)\rightarrow 0$ for every $g\in C(\tx)$.

Take $h\in C(\txp)$. Then, $h|_\hxp\circ
\hf=\hf^*(h|_\hxp)=g|_\hx+r$, with $g\in C(\tx)$ and $r\in
B_0(\hx)$. Thus:
$$\dd_h(\hf(x_\lambda),\hf(y_\lambda))=|h(\hf(x_\lambda))-h(\hf(y_\lambda))|=|h\circ\hf(x_\lambda)-h\circ
\hf(y_\lambda)|=$$
$$|g(x_\lambda)+r(x_\lambda)-g(y_\lambda)-r(y_\lambda)|\leq
|g(x_\lambda)-g(y_\lambda)|+|r(x_\lambda)|+|r(y_\lambda)|=$$
$$\dd_g(x_\lambda,y_\lambda)+|r(x_\lambda)|+|r(y_\lambda)|\rightarrow 0$$
\end{proof}

\begin{lema}\label{caraclfiguallg} Let $\xiv$ and $\xpiv$ be compactification packs and $\hf,\hg:\hx\rightarrow\hxp$ proper maps
such that $L(\hf)$ is defined. Suppose $\ddd'$ is a family of
pseudometrics of $\txp$ which generes its topology. Then,
\begin{enumerate}[a) ]
\item $L(\hg)=L(\hf)$.
\item For every $\dd'\in\ddd'$,
$\lim_{z\rightarrow\infty}\dd'(\hf(z),\hg(z))=0$.
\item $(\hf^*-\hg^*)(C(\txp)|_\hxp)\subset B_0(\hx)$.
\end{enumerate}are equivalent.\end{lema}
\begin{proof}
The equivalence between a) and b) is due to Lemma \ref{lemaequil2}.

Let us see that a) implies c). Take $h\in C(\txp)$. Then,
$l\big((\hf^*-\hg^*)(h|_\hxp)\big)=l\big(h|_\hxp\circ
\hf-h|_\hxp\circ \hg\big)=l(h|_\hxp)\circ L(\hf)-l(h|_\hxp)\circ
L(\hg)=0$. Thus, $(\hf^*-\hg^*)(h|_\hxp)\in B_0(\hx)$.

Let us see that c) implies b). Consider the family of pseudometrics
of $\txp$ $\{\dd_h:h\in C(\txp)\}$. Take $h\in C(\txp)$. Then
$h|_\hxp\circ \hf-h|_\hxp\circ \hg=(\hf^*-\hg^*)(h|_\hxp)\in
B_0(\hx)$. Therefore,
$\lim_{z\rightarrow\infty}\big(h(\hf(z))-h(\hg(z))\big)=0$, that is,
$\lim_{z\rightarrow\infty}\dd_h(\hf(z),\hg(z))=0$.\end{proof}

\section{Functors between coarse structures and compactifications}\label{section4}

Given a locally compact space $\hx$, $E\subset\hx\times\hx$ and a
map $\phi:\hx\times\hx\rightarrow\Rset$, everywhere we use the
expression $\lim_{\substack{(x,y)\rightarrow\infty\\(x,y)\in
E}}\phi(x,y)=a$, we mean the limit of $\phi(x,y)$ when
$(x,y)\rightarrow\infty$ in $\hx\times\hx$ and $(x,y)\in E$. That
is, for every neighborhood $U$ of $a$ in $\Rset$, there exist a
compact subset $K$ of $\hx$  such that $\phi(x,y)\in U$ for every
$(x,y)\in E\bs K\times K$.

Note that, if $E$ is relatively compact in $\hx\times\hx$, according
with the definition above, then
$\lim_{\substack{(x,y)\rightarrow\infty\\(x,y)\in E}}\phi(x,y)$ can
be any $a\in \Rset$.

Observe that $\lim_{\substack{(x,y)\rightarrow\infty\\(x,y)\in
E}}\phi(x,y)=a$ if and only if for every net
$\{(x_\lambda,y_\lambda)\}\subset E$ with
$(x_\lambda,y_\lambda)\rightarrow\infty$ we have that
$\phi(x_\lambda,y_\lambda)\rightarrow a$.

\subsection{The $\ceo$ coarse structure. A functor between compactifications and coarse
structures.}\label{section41}

Recall the following definition, see \cite{cg,cg3} and \cite{jpmd}:

\begin{deff}\label{propcoarsecompacto} Let $\xid$ be a compactification pack. The
topological coarse structure $\ce$ over $\hx$ attached to the
compactification $\tx$ is the collection of all $E\subset \hx\times
\hx$ satisfying any of the following equivalent properties:
\begin{enumerate}[a) ]
\item $Cl_{\tx\times\tx}E$ meets $\tx\times\tx\bs\hx\times\hx$ only
in the diagonal of $X\times X$.

\item $E$ is proper and for every net $(x_\lambda,y_\lambda)\subset
E$, if $x_\lambda$ converges to a point $x$ of $X$, then $y_\lambda$
converges also to $x$.

\item $E$ is proper and for every point $x\in X$ and every
neighborhood $V_x$ of $x$ in $\tx$ there exists a neighborhood $W_x$
of $x$ in $\tx$ such that $E(W_x)\subset V_x$.

\item For every net $(x_\lambda,y_\lambda)\subset
E\cup E^{-1}$, if $x_\lambda$ converges to a point $x$ of $X$, then
$y_\lambda$ converges also to $x$.

\item For every point $x\in X$ and every
neighborhood $V_x$ of $x$ in $\tx$ there exists a neighborhood $W_x$
of $x$ in $\tx$ such that $(E\cup E^{-1})(W_x)\subset V_x$.\end{enumerate}\end{deff}

\begin{rmk}
This coarse structure is preproper. Moreover, if $\tx$ is
metrizable, then $\ce$ is proper.\end{rmk}

\begin{rmk} The definition above is Definition 2.28
of \cite{cg}. The equivalences of a)-e) are given in Proposition
2.27 of \cite{cg} (pags. 26-27), together with the the author's
correction in \cite{cg3} and in Remark 6 and Proposition 7 of
\cite{jpmd}.
\end{rmk}

\begin{rmk}
Properties b) and d) can be rewritten in the language of filters:
\begin{enumerate}[a')]
\setcounter{enumi}{1}
\item Consider the proyections $\pi_i:\tx\times\tx\rightarrow\tx$. $E$ is proper and for every filter
$\mathfrak F$ in $\hx\times\hx$ such that $E\in\mathfrak F$ and
$\pi_1(\mathfrak F)\rightarrow x\in X$ we have that $\pi_2(\mathfrak
F)\rightarrow x$. \setcounter{enumi}{3}
\item Consider the projections $\pi_i:\tx\times\tx\rightarrow\hx$. For very filter
$\mathfrak F$ in $\hx\times\hx$ such that $E\cup E^{-1}\in\mathfrak
F$ and $\pi_1(\mathfrak F)\rightarrow x\in X$ we have that
$\pi_2(\mathfrak F)\rightarrow x$.\end{enumerate}
\end{rmk}

Recall the following definition of Wright in \cite{wright} or
\cite{wright3} (see also of Example 2.6 of \cite{cg}, page 22):

\begin{deff} Let $(\hx,\dd)$ be a metric space. The $C_0$ coarse
structure, denoted by $\ceo(\dd)$ or by $\ceo$ when no confusion
arise, is the collection of all subsets $E\subset\hx\times\hx$ such
that for every $\epsil>0$ there exists a compact subset $K$ of $\hx$
such that $\dd(x,y)<\epsil$ whenever $(x,y)\in E\bs K\times
K$.\end{deff}

In \cite{zs} (Proposition 6, pg. 5237) it is proved that if $\xiv$
is a metrizable compactification pack and $\dd$ is a metric of $\tx$
restricted to $\hx$, then the topological coarse structure attached
to $\hx$ and $\ceo(\dd)$ are equal (actually, this proposition is
not expressed in that terms, but the generalization is trivial).
Taking families of pseudometrics we can generalize this result to
all the compactification packs.

\begin{prop}\label{ceodddec}Let $X$ be a locally compact Hausdorff space and let
$\ddd$ a family of pseudometrics of $X$ which generate its topology.
Denote by $\ceo(\ddd)$ the family of all $E\subset X\times X$ such
that: \begin{enumerate}[a) ]
\item $E$ is proper.
\item  For every $\dd\in\ddd$ and any
$\epsil>0$ there exists a compact subset $K$ of $X$ such that
$\dd(x,y)<\epsil$ for every $(x,y)\in E\bs K\times K$ (equivalently,
$\forall \dd\in\ddd,
\lim_{\substack{(x,y)\rightarrow\infty\\(x,y)\in E}}\dd(x,y)=0$).
\end{enumerate}
Then, $\ceo(\ddd)$ is a preproper coarse structure. Moreover, if
$\ddd$ generates $X$'s topology, then property a) follows from
property b).
\end{prop}

\begin{rmk} Then, if $\ddd$ generes $\hx$'s topology, to see
that $E\in\ceo(\ddd)$ we just have to check property b).\end{rmk}

\begin{proof}[Proof of Proposition \ref{ceodddec}]

Let us see that $\ceo(\ddd)$ is a coarse structure. If follows
easily from the definition that $\ceo(\ddd)$ contains the diagonal
and is closed under the formation of inverses and subsets. Let us
see that it is also closed under finite unions and products.

Pick $E,F\in\ceo(\ddd)$. Let us see that $E\cup F\in\ceo(\ddd)$. Let
$\dd\in\ddd$ and $\epsil>0$. Take two compact sets $K_1$ and $K_2$
such that $\dd(x,y)<\epsil$ whenever $(x,y)\in E\bs K_1\times K_1$
and $\dd(x',y')<\epsil$ whenever $(x',y')\in F\bs K_2\times K_2$.
Let $K=K_1\cup K_2$ and let $(x,y)\in (E\cup F)\bs K\times K\subset
(E\bs K_1\times K_1)\cup (F\bs K_2\times K_2)$. Clearly,
$\dd(x,y)<\epsil$. Therefore, $E\cup F\in\ceo(\ddd)$.

Let us see that $E\circ F\in\ceo(\ddd)$. Let $\dd\in\ddd$ and let
$\epsil>0$. Take two compact sets $K_1$ and $K_2$ such that
$\dd(x,y)<\frac{\epsil}{2}$ whenever $(x,y)\in E\bs K_1\times K_1$
and $\dd(x',y')<\frac{\epsil}{2}$ whenever $(x',y')\in F\bs
K_2\times K_2$.

Let $K=K_1\cup K_2\cup \overline{E(K_1)}\cup
\overline{F^{-1}(K_2)}$. Since $E,F$ are proper, $K$ is compact.
Pick $(x,z)\in E\circ F\bs K\times K$. Let $y$ be such that
$(x,y)\in E$ and $(y,z)\in F$. Let us show that
\begin{equation}\label{jasd4fjk3l}
(x,y)\in E\bs K_1\times K_1\textrm{ and }(y,z)\in F\bs K_2\times K_2
\end{equation}
Observe that or $x\not\in K$ either $z\not \in K$. If $x\not \in
K\supset K_1\cup E(K_2)$ then, $x\not\in K_1$ and $y\not\in K_2$ and
we get (\ref{jasd4fjk3l}). If $z\not\in K\supset K_2\cup
F^{-1}(K_1)$ then, $z\not\in K_2$ and $y\not\in K_1$ and we get
(\ref{jasd4fjk3l}). Hence:
$$\dd(x,z)\leq\dd(x,y)+\dd(y,z)<\frac{\epsil}{2}+\frac{\epsil}{2}=\epsil$$
and, finally, $E\circ F\in\ceo(\ddd)$.

By definition, each $E\in\ceo(\ddd)$ is proper. If $K$ is a
relatively compact set, then, clearly, $K\times K\in\ceo(\ddd)$ and
$K$ is bounded in $\ceo(\ddd)$. Thus, $\ceo(\ddd)$ is preproper.

Suppose now that $\ddd$ generates $X$'s topology and $E\subset
X\times X$ satisfies b). Let us prove that $E$ is proper, by using
Proposition \ref{epropiolimites}'s characterization. Let
$(x_\lambda,y_\lambda)\subset E$. Suppose that
$x_\lambda\rightarrow\infty$. Then
$(x_\lambda,y_\lambda)\rightarrow\infty$ and, by b),
$\dd(x_\lambda,y_\lambda)\rightarrow 0$ for every $\dd\in\ddd$.
Then, by Proposition \ref{xinftysiiyinfty},
$y_\lambda\rightarrow\infty$. By symmetry,
$y_\lambda\rightarrow\infty$ implies $x_\lambda\rightarrow\infty$.
\end{proof}

Now, we can generalize the $C_0$ coarse structure:

\begin{deff} Let $\hx$ be a locally compact space and let
$\ddd $ a family of pseudometrics of $\hx$. The $C_0$ coarse
structure of $\hx$ attached to $\ddd$, denoted by $\ceo(\hx,\ddd)$,
or by $\ceo(\ddd)$ when no confusion can arise is the one defined in
Proposition \ref{ceodddec}.\end{deff}

\begin{rmk}
If $\xiv$ is a compactification pack and $\ddd$ a family of
pseudometrics of $\tx$, by $(\hx,\ceo(\ddd))$ we mean
$(\hx,\ceo(\ddd|_\hx))$.
\end{rmk}
\begin{rmk}
If  $\{\ddd_i\}$ is a set of families of pseudometrics of $X$ then,
$\ceo\left(\bigcup_{i\in I}\ddd_i\right)=\bigcap_{i\in
I}\ceo(\ddd_i)$.
\end{rmk}

\begin{lema}\label{propioc0}Let $\hx$ be a locally compact space. Consider $\ddd=\{\dd_f:f\in C_0(\hx)\}$
and suppose $E\subset\hx\times\hx$. Then, $E$ is proper if and only
if $E\in\ceo(\ddd)$.\end{lema}

\begin{proof} If $E\in\ceo(\ddd)$, clearly $E$ is proper (it follows
from property b) of Proposition \ref{ceodddec}, because $\ddd$
generates $\hx$'s topology).

Suppose now $E$ is proper and choose $f\in C_0(\hx)$. Take
$\{(x_\lambda,y_\lambda)\}\subset E$ with
$(x_\lambda,y_\lambda)\rightarrow\infty$. Since Remark
\ref{obsepropioinfty} shows that $x_\lambda\rightarrow\infty$ and
$y_\lambda\rightarrow\infty$, we have that
$\lim_{\substack{(x,y)\rightarrow\infty\\(x,y)\in E}}\dd_f(x,y)=
\lim_{\substack{x\rightarrow\infty\\y\rightarrow\infty\\(x,y)\in
E}}|f(x)-f(y)|= |0-0|=0$. Then, $E\in\ceo(\ddd)$.\end{proof}

The following Proposition generalizes Proposition 6, pg. 5237:

\begin{prop}\label{coarsecomptc0gen} Let $\xiv$ be a compactification pack
and $\ddd $ a family of pseudometrics which genere $\tx$'s topology.
Then, the topological coarse structure over $\hx$ attached to the
compactification $\tx$ is the $C_0$ coarse structure over $\hx$
attached to $\ddd|_\hx$.
\end{prop}
\begin{proof}
Denote by $\ce$ and $\ceo$ the topological coarse structure attached
to $\tx$ and the $C_0$ coarse structure over attached to $\ddd|_\hx$
respectively.

Let $E\in \ce$ symmetric. Take $\dd\in\ddd$ and $\epsil>0$. For
ever$z\in X$, $\BB_{\dd}(z,\frac{\epsil}{2})$ is a neighborhood of
$z$. Then, by property b) of Definition \ref{propcoarsecompacto},
there exist an open neighborhood $V_z$ of $z$  contained in
$\BB_{\dd}(z,\frac{\epsil}{2})$ such that $E(V_z)\subset
\BB_{\dd}(z,\frac{\epsil}{2})$. Consider the compact set
$K=\tx\bs\bigcup_{z\in X} V_z$. Pick $(x,y)\in E\bs K\times K$.
Suppose, without loss of generality, that $x\not\in K$. Then, $x\in
V_z\subset \BB_{\dd}(z,\frac{\epsil}{2})$ for any $z$ and, hence,
$y\in E(V_z)\subset \BB_{\dd}(z,\frac{\epsil}{2})$. Thus:
$$\dd(x,y)\leq \dd(x,z)+\dd(z,y)< \epsil$$
and $\ce\subset \ceo$.\par

Let $E\in \ceo$. Let us see that $E$ satisfies property a) of
Definition \ref{propcoarsecompacto}.

Pick $(x,y)\in \left(\adh_{\tx\times \tx} E\right) \bs (\hx\times
\hx)$ and take $\{(x_\lambda,y_\lambda)\}\subset E$ such that
$(x_\lambda,y_\lambda)\rightarrow (x,y)$. Thus,
$(x_\lambda,y_\lambda)\rightarrow\infty$ in $\hx\times\hx$.
Consequently, $\dd(x,y)=\lim\dd(x_\lambda,y_\lambda)=0$ for every
$\dd\in\ddd$. Then, $x=y$, hence $E\in\ce$ and $\ceo\subset
\ce$.\end{proof}

\begin{deff} Let $\xid$ be a compactification pack and $\ddd$ a family of pseudometrics which genere $\tx$'s topology.
The $C_0$ coarse structure over $\hx$ attached to $\xid$, denoted by
$\ceo\xid$ or by $\ceo$ when no confusion can arise, is the
topological coarse structure attached to the compactification $\tx$,
i. e. the $C_0$ coarse structure attached to $\ddd|_\hx$.\end{deff}

Taking into account that if $\xiv$ is a compactification pack and
that if $C(\tx)=\langle F\rangle$ then $\{\dd_f:f\in F\}$ generates
$\tx$'s topology, we have that:

\begin{cor}\label{invhigson1} Let $\xiv$ be a compactification pack.
Consider $F\subset C(\tx)$ such that $C(\tx )=\langle F\rangle$.
Then, $$\ceo(\tx)=\{E\subset\hx\times\hx:\forall f\in
F,\lim_{\substack{(x,y)\rightarrow\infty\\(x,y)\in
E}}\dd_f(x,y)=0\}$$.\end{cor}

\begin{cor}\label{invhigson2} Let $\xiv$ be a compactification pack.
Then, $$\ceo(\tx)=\{E\subset\hx\times\hx:\forall f\in
C(\tx),\lim_{\substack{(x,y)\rightarrow\infty\\(x,y)\in
E}}\dd_f(x,y)=0\}$$.
\end{cor}

\begin{rmk}Observe the similarity between the characterization of the
topological coarse structure attached to a compactification $\tx$ by
means of $C(\tx)$ given in Corollary \ref{invhigson2} and the
definition of the algebra of continuous functions of the Higson-Roe
compactification attached to a preproper coarse structure $\ce$ (see
section \ref{section42}, below):
$$C_h(\ce)=\{f\in C_h(\hx):\forall
E\in\ce,\lim_{\substack{(x,y)\rightarrow\infty\\(x,y)\in
E}}\dd_f(x,y)=0\}$$
\end{rmk}

\begin{exa}\label{exaalexandrov}
Let $\hx$ be locally compact and let $A$ be its Alexandrov
compactification. Then, $E\in \ceo(A)$ if and only if $E$ is proper
(examples 2.8 and 2.30 of \cite{cg}, pgs 22 y 27). It can be also
easily proved using Corollary  \ref{invhigson2} and Lemma
\ref{propioc0} and taking into account that
$C(A)|_\hx=C_0(\hx)+\langle 1\rangle$.

\end{exa}
\begin{exa}

\label{ejalexandrovstonecechb} If $\hx$ is locally compact and
$\sigma$-compact space and $E\subset\hx\times\hx$, then $E\in
\ceo(\beta\hx)$ if and only if $E\subset K\times K\cup\Delta$ for
any compact subset $K$ of $\hx$.

If $E\subset  K\times K\cup \Delta$ for any compact subset $K$ of
$\hx$, clearly $E\in\ceo(\beta\hx)$, because it is preproper.
Suppose now that $E$ is a symmetric and proper subset of
$\hx\times\hx$ with $\Delta\subset\hx$ such that $E\not\in K\times
K\cup\Delta$ for every compact subset $K$ of $\hx$ and let us see
that $E\not\in\ceo(\beta\hx)$.

Let $\{K_n\}_{n=1}^\infty$ be a family of compact subsets of $\hx$
with $K_1\subset\mathring K_2\subset K_2\subset\mathring K_3\subset
K_3\subset\dots$ whose union is $\hx$.

Let us define by induction a sequence $\{(x_k,y_k)\}\subset E$ and
$\{n_k\}$ as follows. For $k=1$, take any $(x_1,y_1)\in E\bs
\Delta$. Then, $x_1\neq y_1$. Let $n_1$ such that $x_1,y_1\in
\mathring K_{n_1}$. Suppose that $x_k,y_k,n_k$, are defined with
$x_k\neq y_k$ and $x_k,y_k\in \mathring K_{n_k}$. Since $E(K_{n_k})$
is relatively compact, there exists $(x_{k+1},y_{k+1})\in E\bs
\big(E(K_{n_k})\times E(K_{n_k})\cup \Delta\big)$. Clearly,
$x_{k+1}\neq y_{k+1}$ and $x_{k+1},y_{k+1}\not\in K_{n_k}$. Let
$n_{k+1}>n_k$ be such that $x_{k+1},y_{k+1}\in \mathring
K_{n_{k+1}}$.

Put $K_ {n_0}=\varnothing$. For every $k$, take an open neighborhood
$U_k$ of $x_k$ such that $\overline U_k\subset \left(\mathring
K_{n_k}\bs K_{n_{k-1}}\right)\bs\{y_k\}$ and let $\mu_k:\overline
U_k\rightarrow [0,1]$ be a continuous function with $\mu_k(x_k)=1$
and $\mu_k(\delta U_k)=0$.

Clearly, $\overline U_k\cap \overline U_{k'}$ for every $k\neq k'$
and $\{y_j\}_{j=1}^\infty\cap \overline U_k=\varnothing$ for every
$k$. Let $f:\hx\rightarrow [0,1]$ be the continuous function such
that $f|_{\overline U_k}=\mu_k$ for every $k$ and
$f|_{\hx\bs\bigcup_{k=1}^\infty\overline U_k}=0$. Then, $f\in
C(\beta\hx)|_\hx$ with $f(x_k)=1$ and $f(y_k)=0$ for every $k$.
Thus, $\dd_f(x_{n_k},y_{n_k})=1$ for every $k$, with
$(x_{n_k},y_{n_k})\subset E$ and
$(x_{n_k},y_{n_k})\rightarrow\infty$, hence $E\not\in
\ceo(\beta\hx)$.

\end{exa}

\begin{exa}
 $(\Rset,\ceo(\beta\Rset))$ is not proper. Indeed, the coarse structure described in Example
\ref{ejalexandrovstonecechb} doesn't have neighborhoods of the
diagonal.
\end{exa}

\begin{exa}

\label{ejalexandrovstonecechd} Let $\hx$ and $\hxp$ be locally
compact Hausdorff spaces such that  $\hx$ is $\sigma$-compact and
consider a preproper coarse structure $\cep$ over $\hxp$. Then
$\hf:(\hx,\ceo(\beta\hx))\rightarrow(\hxp,\cep)$ is coarse if and
only if $\hf:\hx\rightarrow\hxp$ is biproper.

If $\hf$ is coarse then it is biproper, due to Example
\ref{exabiproper}. Suppose now that $\hf$ is biproper. Since
$\ceo(\beta\hx)$ and $\cep$ are preproper, $\hf$ is coarsely proper.
Let $E$ be a controlled set of $\ceo(\beta\hx)$. By Example
\ref{ejalexandrovstonecechb}, there exist a compact subset $K$ of
$\hx$ such that $E\subset K\times K\cup\Delta$.  Therefore,
$\hf\times\hf(E)\subset\hf(K)\times\hf(K)\cup\Delta\in\cep$. Then,
$\hf$ is coarse.
\end{exa}

\begin{prop} Let $\xiv$ be a compactification pack. Let $F\subset C(X)$
be such that $C(X)=\langle F\rangle$. For every $f\in F$, let
$\tf:\tx\rightarrow \BBC(0,\|f\|)\subset\Rset$ an extension of
$\tx$. Let $\widetilde F=\{\tf:f\in F\}$ and put $D=\{D_\tf:\tf\in
\widetilde F\}$. Then:
$$\ceo(\tx)=\ceo(\ddd)=\{E\subset\hx\times\hx\textrm{ proper }:\forall \tf\in \widetilde F\lim_{\substack{(x,y)\rightarrow\infty\\(x,y)\in E}}\dd_f(x,y)=0\}$$
\end{prop}
\begin{proof}
By Lemma \ref{propioc0}, $E\subset\hx\times\hx$ is proper if and
only if $E\in\ceo(\{\dd_f:f\in C_0(\hx)\})$. We may suppose that
$C_0(\hx)\subset C(\tx)$ by defining $g|_X=0$. Then $C(\tx)=\langle
C_0(\hx)\cup \widetilde F\rangle$.

Thus, by corollary \ref{invhigson1}, $E\in\ceo(\tx)$ if and only if
for every $g\in C_0(\hx)$,
$\lim_{\substack{(x,y)\rightarrow\infty\\(x,y)\in E}}\dd_g(x,y)=0$
and for every $\tf\in \widetilde F$,
$\lim_{\substack{(x,y)\rightarrow\infty\\(x,y)\in
E}}\dd_\tf(x,y)=0$, i.e., if and only if $E$ is proper and for every
$\tf\in\widetilde F$,
$\lim_{\substack{(x,y)\rightarrow\infty\\(x,y)\in
E}}\dd_\tf(x,y)=0$.\end{proof}

Definition \ref{propcoarsecompacto} tells us how to map
compactification packs into coarse structures. To define a functor,
we have to say how to map compactifications pack's morphisms into
coarse maps. But before, we need to define a reasonable category of
morphisms between compactification packs:

\begin{deff} Let $\xiv$ and $\xpiv$ be compactification packs. We say
that $\tf:\tx\rightarrow\txp$ is asymptotically continuous if
$\tf(\hx)\subset\hxp$, $\tf|_\hx:\hx\rightarrow\hxp$ is biproper and
$\tf=T(\tf|_\hx)$.
\end{deff}

\begin{rmk}\label{obsasintotica} If $\tf:\tx\rightarrow\txp$ is asymptotically continuous, then $\tf(X)\subset X'$ and $\tf|_X:X\rightarrow X'$ is continuous (see Lemma \ref{infinitopropia}).
 $\tf:\tx\rightarrow\txp$ is asymptotically continuous if and only $\tf=T(\hf)$, for certain biproper function $\hf:\hx\rightarrow\hxp$.

\end{rmk}

\begin{rmk}\label{contasimptcont}
If $\tf:\tx\rightarrow\txp$ is continuous, then $\tf$ is
asymptotically continuous if and only if $\tf(\hx)\subset\hxp$ and
$\tf(X)\subset X'$ (see Lemma \ref{infinitopropia} and Example
\ref{ejsbipropiaa}).
\end{rmk}

\begin{rmk} Using Remark \ref{obsasintotica} and Lemma \ref{ltfunctoriales}, it is easy to check that the
composition of asymptotically continuous maps is asymptotically
continuous. Moreover, the identity is an asymptotically continuous
map. Then, the compactification packs with the asymptotically
continuous maps form a category.
\end{rmk}

\begin{prop}\label{ldefcoarse1} Let $\xiv$ and $\xpiv$ be compactification
packs. Suppose $\hf:\hx\rightarrow\hxp$ is biproper function such
that $L(\hf)$ is defined. Then, $\hf:(\hx,\ceo(\tx))\rightarrow
(\hxp,\ceo(\txp))$ is coarse.\end{prop}

\begin{proof}We will use Proposition \ref{caractdel}'s characterization of $L(\hf)$'s existence and Proposition
\ref{coarsecomptc0gen}'s characterization of $\ceo(\tx)$ and
$\ceo(\txp)$. Let $\ddd$ and $\ddd'$ be two families of
pseudometrics which genere the topologies of $\tx$ and $\txp$
respectively.

Since $\hf$ is biproper and $\ceo(\tx)$ and $\ceo(\txp)$ are
preproper, $\hf$ is coarsely proper.

Choose $E\in\ceo(\tx)$. Let us see that $ \hf\times
\hf(E)\in\ceo(\txp)$. Fix $\dd'\in\ddd$ and $\epsil>0$. Pick
$(x'_\lambda,y'_\lambda)\subset \hf\times \hf(E)$ with
$(x'_\lambda,y'_\lambda)\rightarrow\infty$. Take
$x_\lambda,y_\lambda\in \hx$ such that $x'_\lambda= \hf(x_\lambda)$
and $y'_\lambda= \hf(y_\lambda)$ for every $\lambda$. Since $ \hf$
is biproper and $(\hf(x_\lambda),\hf(y_\lambda))\rightarrow \infty$,
we have that $(x_\lambda,y_\lambda)\rightarrow\infty$.

Since $E\in\ceo(\tx)=\ceo(\ddd)$, for every $\dd\in\ddd$,
$\dd(x_\lambda,y_\lambda)\rightarrow 0$. Since $L( \hf)$ is defined,
$\dd'(x'_\lambda,y'_\lambda)=\dd'(\hf(x_\lambda),
\hf(y_\lambda))\rightarrow 0$ for every $\dd'\in\ddd'$. Therefore, $
\hf\times
 \hf(E)\in\ceo(\ddd')=\ceo(\tx)$.\end{proof}

\begin{rmk} In particular, if $\hf:\hx\rightarrow\hxp$ is a
continuous and proper function which extends to a function
$\tf:\tx\rightarrow\txp$, then $\hf$ is coarse.\end{rmk}

\begin{prop}\label{ligualcercanas} Let $\xiv$ and $\xpiv$ be compactification packs. Suppose
$\hf_1,\hf_2:\tx\rightarrow\txp$ are asimptotic maps such that
$L(\hf_1)=L(\hf_2)$. Then $\hf_1$ and $\hf_2$ are close in
$\ceo(\txp)$.\end{prop}

\begin{proof}We will use Proposition
\ref{invhigson2}'s characterization of $\ceo(\txp)$ and Proposition
\ref{caraclfiguallg}'s of $L(\hf_1)=L(\hf_2)$. Let $\ddd'$ be a
family of pseudometrics which generes $\txp$'s topology.

Let us see that the set $E=\{(\hf_1(x),\hf_2(x)):x\in\hx\}$ is in
$\ceo(\txp)$. Since $\hf_1$ and $\hf_2$ are biproper,
$(\hf_1(x),\hf_2(x))\rightarrow \infty$ if and only if
$x\rightarrow\infty$. Then, for every $\dd'\in\ddd'$:
$$\lim_{\substack{(y_1,y_2)\rightarrow\infty\\(y_1,y_2)\in
E}}\dd'(y_1,y_2)=\lim_{\substack{(\hf_1(x),\hf_2(x))\rightarrow\infty\\x\in\hx}}\dd'(\hf_1(x),\hf_2(x))=$$
$$\lim_{\substack{x\rightarrow\infty\\x\in\hx}}\dd'(\hf_1(x),\hf_2(x))=0$$
Therefore, $E\in\ceo(\txp)$ and hence, $\hf_1$ and $\hf_2$ are
close.\end{proof}

We can rewrite Propositions \ref{ldefcoarse1} and
\ref{ligualcercanas} in the following sense:

\begin{prop}\label{ldefcoarse2} Let $\xiv$ and $\xpiv$ be compactification
packs. Suppose $\tf:\hx\rightarrow\hxp$ is an asymptotically
continuous function. Then, $\tf|_\hx:(\hx,\ceo(\tx))\rightarrow
(\hxp,\ceo(\txp))$ is coarse.

In addition, if $\tf_1,\tf_2:\hx\rightarrow\hxp$ are asymptotically
continuous functions such that $\tf_1|_X=\tf_2|_X$, then
$\tf_1|_\hx$ and $\tf_2|_\hx$ are close in $\ceo(\hxp)$.
\end{prop}

\begin{rmk}\label{functorcompletoc0}
We have a functor from the compactification packs with the
asymptotically continuous functions to the preproper coarse spaces
with the coarse maps. It is given by:

$$\begin{array}{rcl}\xid&\overset{\ceo}{\longrightarrow} & \big(\hx,\ceo(\tx)\big)\\
 \big [\tf:\tx\rightarrow\txp\big ] & \overset{\ceo}{\longrightarrow} &
\big [\tf|_\hx:\hx\rightarrow\hxp\big ]\end{array}$$

This functor is kept if we consider the asymptotically continuous
functions identifying two when they are equal in the corona and the
coarse functions identifying two when they are close. Moreover, with
this identifications, the functor is faithful (see Proposition
\ref{ligualcercanas}).

Furthermore, given the compactification packs $\xid$ and $\xpid$, if
$\tf:\tx\rightarrow\txp$ and $\tgg:\txp\rightarrow\tx$ are
asymptotically continuous map such that $\tf|_X:X\rightarrow X'$ and
$\tgg|_X:X'\rightarrow X$ are topologically inverses, then
$\tf|_\hx$ and $\tgg|_\hxp$ are coarse inverses. Indeed,
$(\tgg\circ\tf)|_X=Id_\tx|_X$ and
$(\tf\circ\tgg)|_{X'}=Id_\txp|_{X'}$, hence
$\tgg|_\hxp\circ\tf|_\hx$ and $Id_\hx$ are close and
$\tf|_\hx\circ\tgg|_\hxp$ and $Id_\hxp$ are close.
\end{rmk}

\subsection{The Higson-Roe compactification. A functor between Coarse Structures and
Compactifications.}\label{section42}

Recall the following notions from \cite{cg}, pags. 29-30. Despite of
there the definitions are done using $\Cset$ functions, here we will
use $\Rset$ ones, because to our purpose, both are equivalent.

If $(\hx,\ce)$ is a preproper coarse space, we say that
$B_h(\hx,\ce)$ is the algebra of all bounded functions
$f:\hx\rightarrow \Rset$ such that
$\lim_{\substack{(x,y)\rightarrow\infty\\(x,y)\in E}}\dd_f(x,y)=0$
and that $C_h(\hx,\ce)$ is the subalgebra of continuous functions of
$B_h(\hx,\ce)$.

Since $C_h(\hx,\ce)$ is a closed subalgebra of the algebra of
bounded a continuous real functions of $\hx$, there exists a
compactification $\kh(\hx)$ of $\hx$, such that
$C(\kh(\hx,\ce))|_\hx=C_h(\hx,\ce)$. This is the compactification of
Higson-Roe. The Higson-Roe corona $\nu(\hx,\ce)$ is the corona of
that compactification, that is $\kh(\hx,\ce)\bs\hx$.

When no confusion can arise, we write $B_h(\ce)$, $C_h(\ce)$,
$\kh(\ce)$, $\nu(\ce)$ or $B_h(\hx)$, $C_h(\hx)$, $\kh(\hx)$,
$\nu(\hx)$.

If $\hf:(\hx,\ce)\rightarrow (\hxp,\cep)$ is a coarse map between
coarse preproper spaces, then $\hf^*(B_h(\hxp))\subset B_h(\hx)$ and
$\hf^*(B_0(\hxp))\subset B_0(\hx)$.

From \cite{cg} we take following results (propositions 2.45 - 2.48,
pags. 32-33). If $\hx$ is a locally compact space, $\ce$ and $\ce'$
are preproper coarse structures over $\hx$ and $K$ and $K'$ are
compactifications of $\hx$ then:

\begin{enumerate}[a) ]
\item If $\ce\leq\ce'$, then $K(\ce)\geq K(\ce')$ and, if $K\leq K'$, then $\ceo(K)\geq \ceo(K')$.
\item $\ce\leq \ceo(\kh(\ce))$ and $K\leq \kh(\ceo(K))$.
\item $\kh(\ce)=\kh(\ceo(\kh(\ce)))$ and $\ceo(K)=\ceo(\kh(\ceo(K)))$.
\item If $K$ is metrizable, then
$\kh(\ceo(K))\approx K$.
\item If $\dd$ is a proper metric of $\hx$ then $\ceo(\kh(\ce_b(\dd)))=\ce_b(\dd)$.
\end{enumerate}

As a corollary:
\begin{enumerate}[a) ]\setcounter{enumi}{5}
\item $K=\kh(\ceo(K))$ if and only if $K=\kh(\ce'')$ for certain preproper coarse structure $\ce''$.
\item $\ce=\ceo(\kh(\ce))$ if and only if $\ce=\ceo(K'')$ for certain compactification $K''$.
\end{enumerate}

Characterization of $\ceo(K)$ given in Proposition \ref{invhigson2},
allows us to prove some results easily. For example, a part of a):
If $K\leq K'$, then $C(K)|_\hx\subset C(K')|_\hx$, hence
$\{\dd_f:f\in C(K)|_\hx\}\subset \{\dd_f:f\in C(K')|_\hx\}$ and
thus, $\ceo(\{\dd_f:f\in C(K)|_\hx\})\supset \ceo(\{\dd_f:f\in
C(K')|_\hx\})$.

\begin{prop}\label{caractexistefunctorhigson}
Let $\hf:\hx\rightarrow\hxp$ a proper map between preproper coarse
spaces. Consider the compactification packs $(\nu \hx,\hx,
\kh(\hx))$ and $(\nu \hxp,\hxp, \kh(\hxp))$. Then $L(\hf)$ is
defined if and only if $\hf^*(C_h(\hxp))\subset
C_h(\hx)+B_0(\hx)$.

In particular, if $\hf$ is continuous, then $L(\hf)$ is defined.
\end{prop}

\begin{proof}
By Proposition \ref{caractdel}, $L(\hf)$ is defined if and only if
$\hf^*(C(\kh(\hxp))|_\hxp)\subset C(\kh(\hx))|_\hx+B_0(\hx)$, that
is, $\hf^*(C_h(\hxp))\subset C_h(\hx)+B_0(\hx)$.

If $\hf$ is continuous, then $\hf^*(C_h(\hxp))\subset
C_h(\hx)\subset B_l(\hx,\kh\hx)=C_h(\hx)+B_0(\hx)$.
\end{proof}

\begin{rmk} Of course, $L(\hf)$ is defined if $\hf$ is continuous
outside a compact $K$. Indeed, if $K'$ is a compact subset of $\hx$
such that $\mathring K'$ contains $K$, for every $h'\in C_h(\hxp)$,
we can define $g\in C_h(\hx)$ such that $h\circ\hf-g\in B_0(\hx)$
---in which case, $\hf^*(h)=g+(h\circ\hf-g)\in C_h(\hx)+B_0(\hx)$---,
as follows: Consider the map $h'\circ \hf|_{\partial K'}:\partial
K'\rightarrow\Rset$. By Tietze extension theorem, there exists a
continuous extension $g_0: K'\rightarrow\Rset$. Put
$g:\hx\rightarrow\Rset$ with $g|_{\hx\bs\mathring K'}=h\circ\hf$ and
$g|_{K'}=g_0$. \end{rmk}

\begin{rmk} Suppose $\hx$ and $\hxp$ are preproper coarse spaces and consider their Higson-Roe compactications.
Clearly, $B_l(\hx,\kh\hx)=C_h(\hx)+B_0(\hx)\subset B_h(\hx)$. For
every coarse map $\hf:\hx\rightarrow\hxp$ we have that
$\hf^*(C_h(\hxp))\subset \hf^*(B_h(\hxp))\subset B_h(\hx)$.

That means that, when $B_l(\hx,\kh\hx)=B_h(\hx)$ we have that
$\hf^*(C_h(\hxp))\subset B_l(\hx,\kh\hx)$ and, by Proposition
\ref{caractexistefunctorhigson}, $L(\hf)$ is defined. This happens
when $\hx$ is a proper coarse space (see in ({sumabhchc0} dn
Proposition \ref{coarl}, below). But there are examples of maps
between preproper coarse spaces $\hf:\hx\rightarrow\hxp$ such that
$B_l(\hx,\kh\hx)\subsetneq B_h(\hx)$ and $L(\hf)$ is not defined
(see Examples \ref{ejemplocoarsenopropiaprevio} and
\ref{ejemplocoarsenopropia}, below).
\end{rmk}

If $\hx$ is a proper coarse space, then we have the following
equality (see Lemma 2.4 of \cite{cg} (pag. 40) and Lemma
\ref{algebrabl}):

\begin{equation}\label{sumabhchc0}
B_h(\hx)=B_l(\hx,\kh \hx)=C_h(\hx)+B_0(\hx)
\end{equation}

Consequently:

\begin{prop}\label{coarl} Let $\hf:\hx\rightarrow\hxp$ be a coarse map between coarse spaces, where $\hx$ is proper and $\hxp$, preproper.
Consider the compactification packs $(\nu \hx,\hx, \kh(\hx))$ and
$(\nu \hxp,\hxp, \kh(\hxp))$. Then, $L(\hf)$ is defined.\end{prop}

\begin{proof} By (\ref{sumabhchc0}),
$\hf^*(C_h(\hxp))\subset\hf^*(B_h(\hxp))\subset B_h(\hx)=
C_h(\hx)+B_0(\hx)$, hence $L(\hf)$ is defined, due to Proposition
\ref{caractexistefunctorhigson}.\end{proof}

\begin{prop}\label{cercanasldef} Let $(\hx,\ce)$ and $(\hxp,\cep)$
be preproper coarse spaces and consider the compactification packs
 $(\nu(\ce),\hx,\kh(\ce))$ and $(\nu(\cep),\hxp,\kh(\cep))$.
Let $\hf,\hg:\hx\rightarrow\hxp$ be coarse an closed in $\hxp$ such
that $L(\hf)$ is defined. Then, $L(\hg)=L(\hf)$.\end{prop}

\begin{proof} If $\hf$ and $\hg$ are close, then the set
$E=\{(\hf(x),\hg(x)):x\in\hx\}\in\cep$. Take $h\in
C(\kh(\cep))|_\hx=C_h(\ce)$. Since $\hf$ and $\hg$ are biproper:
$$0=\lim_{\substack{(y_1,y_2)\rightarrow\infty\\(y_1,y_2)\in
E}}\dd_h(y_1,y_2)=\lim_{\substack{(\hf(x),\hg(x))\rightarrow\infty\\x\in\hx}}\dd_h(\hf(x),\hg(x))=\lim_{x\rightarrow\infty}\dd_h(\hf(x),\hg(x))$$

Hence,
$$0=\lim_{x\rightarrow\infty}\big(h(\hf(x))-h(\hg(x))\big)=\lim_{x\rightarrow\infty}\big(\hf^*(h)-\hg^*(h)\big)(x)$$

Thus, $\hf^*(h)-\hg^*(h)\in B_0(\hx)$, because of Lemma
\ref{caractb0limit}. Therefore,
$(\hf^*-\hg^*)\big(C(\kh(\cep))|_\hx\big)\subset B_0(\hx)$ and, by
Lemma \ref{caraclfiguallg}, $L(\hg)=L(\hf)$.\end{proof}

We can rewrite Propositions \ref{caractexistefunctorhigson} and
\ref{cercanasldef} in the following sense:

\begin{prop}\label{coarl} Let $\hf,\hg:\hx\rightarrow\hxp$ be a coarse map between coarse spaces, where $\hx$ is proper and $\hxp$,
preproper. Then, $T(\hf)$ is an asymptotically continuous map
between the compactification packs $(\nu \hx,\hx, \kh(\hx))$ and
$(\nu \hxp,\hxp, \kh(\hxp))$.

If, moreover, $\hf$ and $\hg$ an closed in $\hxp$, then
$T(\hf)|_{\nu\hx}=T(\hg)|_{\nu\hx}$.\end{prop}

\begin{rmk} \label{functorcompletohigsonroe} Then, the Higson-Roe compactification induces a functor
from the proper coarse spaces with the coarse maps to the
compactification packs with the asymptotically continuous maps. It
is given by:

$$\begin{array}{rcl}(\hx,\ce)&\overset{\kh}{\longrightarrow} & \big(\nu \hx,\hx,\kh(\hx)\big)\\
 \big [\hf:\hx\rightarrow\hxp\big ] & \overset{\kh}{\longrightarrow} &
\big [T(\hf):\kh(\hx)\rightarrow\kh(\hxp)\big ]\end{array}$$

This functor is kept if identify two coarse maps when they are close
and two asymptotically continuous maps when they are equal in the
corona. Moreover, with this identifications, the functor is faithful
(see Proposition \ref{cercanasldef}).

Furthermore, given the proper coarse spaces $(\hx,\ce)$ and
$(\hxp,\cep)$, if $\hf:\hx\rightarrow\hxp$ and
$\hg:\hxp\rightarrow\hx$ are coarse inverses, then
$T(\hf)|_X=L(\hf)$ and $T(\hg)|_{X'}=L(\hg)$ are inverses. Indeed,
$\hg\circ\hf$ and $Id_\hx$ are closed and $\hf\circ\hg$ and
$Id_\txp$ are closed, hence $L(\hg)\circ
L(\hf)=L(\hg\circ\hf)=L(Id_\hx)=Id_X$ and $L(\hf)\circ
L(\hg)=L(\hf\circ\hg)=L(Id_\txp)=Id_{X'}$.
\end{rmk}

\begin{exa}\label{ejemplocoarsenopropia}

Let $\hx=[0,1]\times[0,1)$ and $\tx=\beta\hx$. Let
$\hy=\{0,1\}\times[0,1)$ and $\ty=\{0,1\}\times[0,1]$. Consider the
map $\hf:(\hx,\ceo(\tx))\rightarrow(\hy,\ceo(\ty))$ such that, for
every $(x,t)\in\hx$, $\hf(x,t)=(0,t)$ if $x\leq\frac{1}{2}$ and
$\hf(x,t)=(1,t)$ if $x>\frac{1}{2}$. Then,
$\hf:(\hx,\ceo(\tx))\rightarrow(\hx,\ceo(\ty))$ is a coarse map but
$L_{\kh(\ceo(\tx)),\kh(\ceo(\ty))}(\hf)$ is not defined.

Since $\hf:\hx\rightarrow\hx$ is a biproper map, example
\ref{ejalexandrovstonecechd} shows that
$\hf:(\hx,\ceo(\beta\hx))\rightarrow(\hx,\ceo(\ty))$ is coarse.

$\beta\hx\leq \kh(\ceo(\beta\hx))\leq\beta\hx$, hence
$\beta\hx\approx\kh(\ceo(\beta\hx))$. Since $\ty$ is metrizable,
$\ty\approx \kh(\ceo(\ty))$. To see that
$L_{\kh(\ceo(\tx)),\kh(\ceo(\ty))}(\hf)$ is not defined, we will
prove that $L_{\beta\hx,\ty}(\hf)$ is not defined.

Let $\hx^*=\beta\hx\bs\hx$. Since $\hx$ is connected, $\hx^*$ is.
Suppose $L_{\beta\hx,\ty}(\hf)$ is defined. Consider the ordered set
$[0,1)$ with the usual order and the net
$\{x_t\}_{t\in[0,1)}=\{(0,t)\}_{t\in[0,1)}$. Observe that
$\hf(x_t)\rightarrow (0,1)$. Since $x_t\rightarrow\infty$ in $\hx$,
there exist a subnet $x_\lambda$ such that $x_\lambda\rightarrow
\omega\in \hx^*$. Since $\hf(x_\lambda)$ is a subnet of $\hf(x_t)$,
we have that $\hf(x_\lambda)\rightarrow (0,1)$. Then,
$L(\hf)(\omega)=(0,1)$ and, hence, $(0,1)\in L(\hf)(\hx^*)$. Using a
similar argument, we get that $(1,1)\in L(\hf)(\hx^*)$. Then,
$L(\hf):\hx^*\rightarrow\{(0,1),(1,1)\}$ is continuous and
surjective, in contradiction with the connectedness of $\hx^*$.
\end{exa}

\begin{exa}\label{ejemplocoarsenopropiaprevio}
Let $\hx$, $\tx$, $\hy$, $\ty$ and
$\hf:(\hx,\ceo(\tx))\rightarrow(\hy,\ceo(\ty))$ like in Example
\ref{ejemplocoarsenopropia}. Let us see that
$B_l(\hx,\kh\ceo(\tx))\subsetneq B_h(\hx,\ceo(\tx))$.

Consider the map $g:\hy\rightarrow\Rset$, such that $g(0,t)=0$ and
$g(1,t)=1$ for every $t$. Observe that $g\in
C(\ty)|_\hy=C(\kh\hy)|_\hy\subset B_h(\hy)$. Let $h=g\circ\hf$.
Then, $f(x,t)=0$ if $t\leq \frac{1}{2}$ and $f(x,t)=1$ if
$t>\frac{1}{2}$. Since $\hf$ is coarse, $f=g\circ\hf=\hf^*(g)\in
B_h(\hx)$.

Suppose that $f\in B_l(\hx,\ceo(\tx))$. Then, $l(f)$ is defined.
Using a similar argument like in Example
\ref{ejemplocoarsenopropia}, we get that $l(f):\hx^*\rightarrow
\{0,1\}$ is continuous and surjective, in contradiction with the
connectedness of $\hx^*$, that follows from the connectedness of
$\hx$. Then, $f\in B_l(\hx,\ceo(\tx))$ and
$B_l(\hx,\kh\ceo(\tx))\subsetneq B_h(\hx,\ceo(\tx))$.\end{exa}

The following proposition generalizes Proposition 2.33 of \cite{cg}:

\begin{prop}\label{isomextension}Let $\xiv$ and $\xpiv$ be compactification packs
such that $\tx\approx \kh(\ceo(\tx))$ and
$\txp\approx\kh(\ceo(\txp))$. Let $\hf:\hx\rightarrow\hxp$ be a
continuous and proper map. Then
$\hf:(\hx,\ceo(\tx))\rightarrow(\hxp,\ceo(\txp))$ is coarse if and
only if $\hf$ extends to a continuous function
$\tf:\tx\rightarrow\txp$.

Moreover, if
$\hf_1,\hf_2:(\hx,\ceo(\tx))\rightarrow(\hxp,\ceo(\txp))$ are coarse
and $\tf_1,\tf_2:\tx\rightarrow\txp$ are extensions of
 $\hf_1$ and $\hf_2$ respectively, then $\hf_1$ and $\hf_2$ are
close if and only if $\tf_1|_X=\tf_x|_X$.
\end{prop}

\begin{proof}
Suppose that $\hf$ is extended to a continuous function
$\tf:\tx\rightarrow\txp$. Taking into account \ref{infinitopropia},
that means that $T(\hf)$ is defined. Since it is biproper (se Remark
\ref{obsasintotica}), by Proposition \ref{ldefcoarse1}, $\hf$ is
coarse.

Suppose now that $\hf:(\hx,\ceo(\txp))\rightarrow(\hxp,\ceo(\txp))$
is coarse. Then, by Proposition \ref{caractexistefunctorhigson},
$T(\hf):\kh(\ceo(\tx))\rightarrow\kh(\ceo(\txp))$ is defined. Since
$\tx\approx \kh(\ceo(\tx))$ and $\txp\approx\kh(\ceo(\txp))$,
$T(\hf):\tx\rightarrow\txp$ is defined. Since $\hf$ is continuous,
$T(\hf)$ is continuous. Then, $\hf$ has a continuous extension.

If $\tf_1|_X=\tf_2|_X$ then $L(\hf_1)=L(\hf_2)$, hence by
Proposition \ref{ligualcercanas}, $\tf_1$ and $\tf_2$ are close in
$\ceo(\txp)$. If $\hf_1$ are $\hf_2$ close in $\ceo(\txp)$, by
Proposition \ref{cercanasldef},
$L_{\nu(\ceo(\tx)),\nu(\ceo(\txp))}(\hf_1)=L_{\nu(\ceo(\tx)),\nu(\ceo(\txp))}(\hf_2)$.
Since $\tx\approx \kh(\ceo(\tx))$ and $\txp\approx\kh(\ceo(\txp))$
$L_{\tx,\txp}(\hf_1)=L_{\tx,\txp}(\hf_2)$. Then,
$\tf_1|_X=\tf_2|_X$.
\end{proof}

\begin{rmk} An equivalent way of enunciate Proposition
\ref{isomextension} is: If $\hf:(\hx,\ce)\rightarrow(\hxp,\cep)$ is
a proper and continuous map between preproper coarse spaces such
that $\ce=\ceo(\kh(\ce))$ and $\cep=\ceo(\kh(\cep))$, then $\hf$ is
coarse if and only if $\hf$ extends to a continuous function
$\tf:\kh(\hx)\rightarrow\kh(\hxp)$.

\end{rmk}

\begin{rmk}
It also can be proved by using the equivalence of categories of
Corollary \ref {eqcat1} - Remark \ref{functoracont}.
\end{rmk}

\begin{exa}\label{noexistefcoarse} Let $A=\Nset\cup\{\infty\}$ be the Alexandrov compactification of $\Nset$
and consider the compactification packs $(\{\infty\},\Nset,A)$ and
$(\Nset^*,\Nset,\beta\Nset)$. Then, there is no asymptotically
continuous maps $\tf:\tx\rightarrow\txp$ or, equivalently, there is
no coarse maps $\hf:(\Nset,\ceo(A))\rightarrow
(\Nset,\ceo(\beta\Nset))$.

Suppose such $\hf$ exists. Then, $\hf$ is biproper (see example
\ref{exabiproper}). Since $\Nset$ is not compact, $\hf(\Nset)$ is
not. Then, we may take $\{y_k\}_{k=1}^\infty\subset\hf(\Nset)$ with
$y_k\neq y_{k'}$ for every $k\neq k'$. For every $k$, take $x_k$
such that $\hf(x_k)=y_k$. Clearly, $x_k\neq x_{k'}$ for every $k\neq
k'$.

It is easy to check that the set $E=\{(x_{2m-1},x_{2m}):m\in\Nset\}$
is a proper subset of $\Nset\times\Nset$. Then, by Example
\ref{exaalexandrov}, $E\in\ceo(A)$. But
$\hf\times\hf(E)=\{(y_{2m-1},y_{2m}):m\in\Nset\}$ and, taking into
account Example \ref{ejalexandrovstonecechb}, it is easy to check
that $\hf\times\hf(E)\not\in \ceo(\beta\hx)$. This contradicts our
assumption and such $\hf$ doesn't exist. \end{exa}

\subsection{The Higson-Roe functor. A topological interpretation.}\label{section43}

If $\hf:\hx\rightarrow\hxp$, the Higson-Roe functor
$\nu(\hf):\nu\hx\rightarrow\nu\hxp$ is a continuous map between the
Higson-Roe coronas. Moreover, the functor is kept if we identify two
maps when they are close, that is: if $\hf$ and $\hg$ are close,
then $\nu\hf=\nu\hg$.

Let us describe the definition of $\nu\hf$ given in \cite{cg} (pag.
31, above). In \cite{cg2} there is another equivalent definition,
just for the bounded coarse structure, but easy to generalize taking
Lemma 2.40 of \cite{cg} (pag. 30) into account.

Let us consider the following isomorphisms:
\begin{enumerate}[a) ]
\item $\frac{B_h(\hx)}{B_0(\hx)}\approx \frac{C_h(\hx)}{C_0(\hx)}$.
For all $f\in B_h(\hx)$, the isomorphism maps $f+B_0(\hx)$ to
$f+o(f)+C_0(\hx)$, where $o(f)\in B_0(\hx)$ is a function such that
$f+o(f)$ is continuous. The existence of such $o(f)$ is guaranteed
by the equality (\ref{sumabhchc0}).

\item $\frac{C_h(\hx)}{C_0(\hx)}\approx \frac{C(\kh(\hx))}{I(\nu \hx)}$, where by $I(\nu
\hx)$ we understand the functions of $C(\kh(\hx))$ vanishing in
$\nu \hx$. For all $f\in C_h(\hx)$, the isomporphism maps
$f+C_0(\hx)$ to $\tf+I(\nu \hx)$, where $\tf$ is a continuous extension of
 $f$ to $\kh(\hx)$. In other words, the isomorphism maps
$f+C_0(\hx)$ to $t(f)+I(\nu \hx)$, where $t$ is the total map
attached to $(\nu \hx,\hx, \kh(\hx))$ and $\Rset$.
\item $\frac{C(\kh(\hx))}{I(\nu \hx)}\approx C(\nu(\hx))$. For all $f\in C(\kh(\hx))$, the isomorphism maps
$f+I(\nu \hx)$ to $f|_{\nu \hx}$.
\end{enumerate}

Hence, the composition of this isomorphisms defines an isomorphism
$J:\frac{B_h(\hx)}{B_0(\hx)}\rightarrow C(\nu \hx)$. If
$\hx$ and $\hxp$ are proper coarse spaces and
$\phi:\hx\rightarrow \hxp$ is coarse, then $\phi^*(B_h(\hxp))\subset B_h(\hx)$ and
$\phi^*(B_0(\hxp))\subset B_0(\hx)$. By this way, $\phi$ induces a morphism $\phi^*:\frac{B_h(\hxp)}{B_0(\hxp)}\rightarrow
\frac{B_h(\hx)}{B_0(\hx)}$.

Then, we have a morphism $\varphi:C(\nu \hxp)\approx
\frac{B_h(\hxp)}{B_0(\hxp)}\rightarrow
\frac{B_h(\hx)}{B_0(\hx)}\approx C(\nu \hx)$, given by
$\varphi=J\circ \phi^*\circ J^{-1}$.

By Gelfand-Naimark theorem in its real version or by Theorem 10.6 of
\cite{rcf}, (pag. 142), there exist an unique continuous function
 $\nu \phi:\hx\rightarrow\hxp$ such that $(\nu
\phi)^*=\varphi$. The operator $\phi\rightarrow \nu \phi$ is, in
fact, a functor: the Higson-Roe functor.

Given a coarse map $\phi:\hx\rightarrow\hxp$, between proper coarse
spaces, we have defined the functor
$T(\phi):\kh\hx\rightarrow\kh\hxp$ (see Remark
\ref{functorcompletohigsonroe}). This functor defines another
functor in the Higson-Roe coronas in a natural way, by taking
$T(\phi)|_{\nu \hx}:\nu \hx\rightarrow\nu\hxp$. But $T(\phi)|_{\nu
\hx}=L(\phi)$, so the functor is $L(\phi):\nu
\hx\rightarrow\nu\hxp$.

There is a natural question: Are $\nu\phi$ and $L(\phi)$ the same
functor? The following theorem will prove that:

\begin{teo}\label{higsoneslimite} Let $\hx$ and $\hxp$ be proper coarse spaces.
Consider the compactification packs $(\nu \hx,\hx, \kh(\hx))$ and
$(\nu \hxp,\hxp, \kh(\hxp))$. Then, for every coarse map
$\phi:\hx\rightarrow\hxp$ we have that $\nu \phi=L(\phi)$.\end{teo}

\begin{proof}

Let $f\in B_h(\hx)$. The isomorphism
$\frac{B_h(\hx)}{B_0(\hx)}\approx\frac{C_h(\hx)}{C_0(\hx)}$ maps
$f+B_0(\hx)$ to $f+o(f)+C_0(\hx)$, the isomophism
$\frac{C_h(\hx)}{C_0(\hx)}\approx \frac{C(\kh(\hx))}{I(\nu \hx)}$
maps $f+o(f)+C_0(\hx)$ to $t\big(f+o(f)\big)+I(\nu \hx)$ and the
isomorphism
 $\frac{C(\kh \hx)}{I(\nu \hx)}\approx C(\nu \hx)$ maps
$t\big(f+o(f)\big)+I(\nu \hx)$ to $t\big(f+o(f)\big)|_{\nu
\hx}=l\big(f+o(f)\big)=l(f)+l(o(f))=l(f)+0=l(f)$. Hence, the isomorphism
$J:\frac{B_h(\hx)}{B_0(\hx)}\rightarrow C(\nu\hx)$ is given by the map
$J(f+B_0(\hx))=l(f)$, where $l$ is the limit function attached to
$(\nu \hx,\hx, \kh(\hx))$ and $\Rset$.

Consequently, the isomorphism, $J^{-1}:C(\nu \hxp)\rightarrow
\frac{B_h(\hxp)}{B_0(\hxp)}$, is given by
$J^{-1}(f)=j(f)+B_0(\hxp)$, where $j(f)\in B_h(\hxp)$ is any
function such that $l(j(f))=f$.

Let $\phi:\hx\rightarrow\hxp$ be a coarse map. Let $f\in C(\nu
\hxp)$. By definition,
$(\nu\phi)^*(f)=J(\phi^*(J^{-1}(f)))=J(\phi^*(j(f)+B_0(\hxp)))=J\big(j(f)\circ\phi+B_0(\hx)\big)=l(j(f)\circ\phi)=l(j(f))\circ
L(\phi)=f\circ L(\phi)=(L(\phi))^*(f)$. Then,
$(\nu\phi)^*=(L(\phi))^*$, hence $\nu\phi=L(\phi)$.
\end{proof}

\begin{deff}We extend the Higson-Roe functor in the following way:
To every map $\hf:(\hx,\ce)\rightarrow(\hxp,\cep)$ between preproper
coarse spaces, if it is defined, $\nu\hf:\nu\hx\rightarrow\nu\hxp$ is the map
$\nu\hf=L_{\kh\hx,\kh\hxp}\hf$.\end{deff}

\begin{rmk}\label{functorcompletonuydelta}Let us consider the Higson-Roe functor, from the proper
coarse spaces with the coarse maps to the compact Hausdorff spaces
with the continuous functions:
$$\begin{array}{rcl}(\hx,\ce)&\overset{\nu}{\longrightarrow} & \nu \hx\\
 \big [\hf:\hx\rightarrow\hxp\big ] & \overset{\nu}{\longrightarrow} &
\big [L(\hf):\nu(\hx)\rightarrow\nu(\hxp)\big ]\end{array}$$

And consider the trivial functor ``corona'', from the
compactification packs with the asymptotically continuous functions
to the compact Hausdorff spaces with the continuous maps, given by:
$$\begin{array}{rcl}\xid&\overset{\delta}{\longrightarrow} & X\\
 \big [\tf:\tx\rightarrow\txp\big ] & \overset{\delta}{\longrightarrow} &
\big [\tf|_X:X\rightarrow X'\big ]\end{array}$$

Consider also the functor $\kh$ defined in Remark
\ref{functorcompletohigsonroe}. It is easy to check that the
Higson-Roe functor is the composition of the functors $\kh$ and
$\delta$:
$$\nu=\delta\circ\kh$$

This functors are kept if identify two coarse maps when they are
close and two asymptotically continuous maps when they are equal in
the corona. Moreover, with this identifications, this functors are
faithful (see Proposition \ref{cercanasldef}).
\end{rmk}

\subsection{Categories.}\label{section44}

Let us summarize all the information given here about the morphisms
$\ceo$, $\kh$, $\nu$ and $\delta$ described in Remarks
\ref{functorcompletoc0}, \ref{functorcompletohigsonroe} and
\ref{functorcompletonuydelta} (take into account properties c)-g)
given in section \ref{section42}, above):

\begin{teo}\label{pseudoinvdiag}
Consider the category of proper coarse space with the coarse maps
(represented by $(\hx,\ce)$), the category of the compactification
packs such that the attached $C_0$ coarse structure is proper with
the asymptotically continuous maps (represented by $\xid$) and the
category of the compact Hausdorff spaces with the continuous maps
(represented by $X$). Consider the functors described on Remarks
\ref{functorcompletoc0}, \ref{functorcompletohigsonroe} and
\ref{functorcompletonuydelta}.

Then, $\ceo$ and $\hx$ are pseudoinverses and the following diagram
is commutative:
$$\xymatrix@C=21mm@R=21mm{(\hx,\ce) \ar@{->}@<-1.2ex>[r]_{\kh}^{\hspace{-15pt}\textrm{pseudoinverse}}
\ar@{->}^{\,\nu\,\,\,\,\,\,\,\,\,\,\,\,\,\,\,\circlearrowright} [rd]
&\xiv\ar@{->}@<-1.2ex>[l]_{\ceo} \ar@{->}[d]^{\delta}\\&X }$$

Moreover, this functors are preserved if we consider the coarse maps
identifying two when they are closed and the asymptotically
continuous functions identifying two when they are equal in the
corona. And, with this identifications, the functors are faithful.
\end{teo}

\begin{cor}\label{eqcat1}
Consider the categories and functors of the proposition above, with
the extra conditions:
\begin{itemize}
\item The coarse spaces $(\hx,\ce)$ are such that
$\ce=\ceo(\kh(\ceo))$ and the compactification packs $\xid$ are such
that $\tx\approx\kh(\ceo(\tx))$.
\item We identify two coarse maps when they are close and we
identify tho asymptotically continuous maps when they are equal in
the corona.
\end{itemize}
Then, $\ceo$ and $\kh$ is a equivalence of categories, the one
inverse of the other, $\nu$ and $\delta$ are faithful and the
diagram is commutative:
$$\xymatrix@C=21mm@R=21mm{(\hx,\ce) \ar@{->}@<-1.2ex>[r]_{\kh}^{\hspace{-15pt}\,\,\,\,\,\,\,\,\cong}
\ar@{->}^{\,\nu\,\,\,\,\,\,\,\,\,\,\,\,\,\,\,\circlearrowright} [rd]
&\xiv\ar@{->}@<-1.2ex>[l]_{\ceo} \ar@{->}[d]_{\delta}\\&X }$$
\end{cor}

\begin{rmk}\label{functoracont} Proposition \ref{pseudoinvdiag} and Corollary \ref{eqcat1}
are true if we consider all the preproper coarse spaces with the
coarse and continuous maps and all the compactification packs with
the asymptotically continuous and continuous maps (that is, in
$\xid$, continuous maps such that $\tf(\hx)\subset\hxp$  and
$\tf(X)\subset X'$ (see Remark \ref{contasimptcont}).\end{rmk}

\begin{rmk}\label{cuandonueseqcat} According to Corollary \ref{eqcat1}, $\ceo$ and $\kh$
induce an equivalence of categories and $\nu$ and $\delta$ are
faithful functors.

Moreover, $\delta$ (and consequently $\nu$) is trivially a dense
functor, because if $X$ is a compact space, then
$\xid=(X\times\{0\},X\times(0,1],X\times[0,l])$ is a
compactification pack such that $X\approx X\times\{0\}$ (moreover,
if $X$ is metrizable, $\xid$ is).

To be a equivalence of categories, $\nu$ or $\delta$ just should be
full. But generally they are not, as we can see in Example
\ref{nunoespleno}.
\end{rmk}

\begin{exa}\label{nunoespleno} Consider $\Nset$, its Alexandrov compactification
$A=\Nset\cup\{\infty\}$ and the compactification packs
$(\{\infty\},\Nset,A)$ and $(\Nset^*,\Nset,\beta\Nset)$. All the
maps $f:\{\infty\}\rightarrow \Nset^*$ are continuous, but from
Example \ref{noexistefcoarse} we get that there is no
$\hf:(\Nset,\ceo(A))\rightarrow (\Nset,\ceo(\beta\Nset))$ coarse
with $\nu\hf=f$ and no $\tf:A\rightarrow\beta\Nset$ asymptotically
continuous with $\tf|_{\{\infty\}}=f$.\end{exa}

If $\xid$ is a metrizable compactification pack, Property c) in
section \ref{section42}, above, tell us that
$\tx\approx\kh(\ceo(\tx))$. By this reason, $\xid$ and
$(\hx,\ceo(\tx))$ satisfies Corollary \ref{eqcat1}. As we will see
in Theorem \ref{equivmetrizable}, in this case, the equivalences of
categories is stronger. But before, we need the following technical
lemma:

\begin{lema}\label{tecniconuespleno} Let $\xid$ be a metrizable compactification pack and consider the
compactification pack $(X\times\{0\},X\times(0,1],X\times [0,1])$
and the homeomorphism $h:X\rightarrow X\times\{0\}$, $x\rightarrow
(x,0)$. Then there exist asymptotically continuous functions
$\tf:\tx\rightarrow X\times[0,1]$ and $\tgg:X\times[0,1]\rightarrow
\tx$ such that $\tf|_X=h$ and
$\tgg|_{X\times\{0\}}=h^{-1}$.\end{lema}

\begin{proof}Let $\dd_0$ be
a metric on $\tx$. Consider the metric $\dd=\frac{1}{k}\dd_0$ on
$\tx$, where $k=\sup_{x\in\tx}\dd_0(x,X)$.

Consider $\ceo=\ceo(\tx)$ and $\ceop=\ceo(X\times[0,1])$. By
Proposition 38 of \cite{jpmd}, pag. 109, there exists two coarse
equivalences $f:(\hx,\ceo)\rightarrow(X\times(0,1],\ceop)$ and
$g:(X\times(0,1],\ceop)\rightarrow(\hx,\ceo)$, the one inverse of
the other, satisfying:
\begin{itemize}
\item For every $x\in
\hx$ $f(x)=(z,t)$, with $t=\dd(x,X)$, $z\in X$ and $\dd(x,z)=t$.

\item For every $(z,t)\in X\times(0,1]$, $g(z,t)=y$ with $y\in \tx\bs\BB(X,t)$ and
$\dd(y,z)=\dd(z,\tx\bs\BB(X,t))$.\end{itemize}

Let us see that $L(\hf)=h$. Let $x\in X$ and $\{x_n\}\subset\hx$
such that $x_n\rightarrow x$. Put $\hf(x_n)=(z_n,t_n)$. Then $0\leq
t_n\leq \dd(x_n,X)\leq\dd(x_n,x)\rightarrow 0$ and hence,
$t_n\rightarrow 0$. Moreover,
$\dd(z_n,x)\leq\dd(z_n,x_n)+\dd(x_n,x)=t_n+\dd(x_n,x)\rightarrow 0$
and we get $x_n\rightarrow x$.

Therefore, $L(\hf)(x)=\lim \hf(x_n)=\lim (z_n,t_n)=(x,0)=h(x)$ and
we get $L(\hf)=h$.

$\nu=L$ is a functor and, since $\hf$ and $\hg$ coarse inverses,
they are inverses when we identify two maps when they are close.
Thus, $L(\hg)=\L(\hf)^{-1}=h^{-1}$. Then, by Proposition
\ref{coarl}, $T(\hf)$ and $T(\hg)$ satisfy the desired
properties.\end{proof}

\begin{teo}\label{equivmetrizable}
Consider the category of $C_0$ coarse spaces attached to a
completely bounded metric with the coarse maps identifying two when
they are close (represented by $(\hx,\ce)$), the category of the
metrizable compactification packs with the asymptotically continuous
maps identifying two when they are equal in the corona (represented
by $\xid$) and the category of the compact Hausdorff spaces with the
continuous maps (represented by $X$). Consider the functors
described on Remarks \ref{functorcompletoc0},
\ref{functorcompletohigsonroe} and \ref{functorcompletonuydelta}.

Then, $\ceo$, $\hx$, $\nu$ and $\delta$ are coarse equivalences,
being $\ceo$ and $\hx$ the one inverse of the other, and  the
following diagram is commutative:
$$\xymatrix@C=21mm@R=21mm{(\hx,\ce) \ar@{->}@<-1.2ex>[r]_{\kh}^{\hspace{-15pt}\,\,\,\,\,\,\,\,\cong}
\ar@{->}^{\,\nu\,\,\,\,\,\,\,\,\,\,\,\,\,\,\,\circlearrowright}|{\,\cong\,}
[rd] &\xiv\ar@{->}@<-1.2ex>[l]_{\ceo}
\ar@{->}[d]_{\delta\,\,}|{\,\cong\,}\\&X }$$

\end{teo}

\begin{proof} A metric $\dd$ in a locally compact metric $\hx$ space is totally bounded if and only if
it is the metric of a (metrizable) compactification $\tx$ of $\hx$
restricted to $\hx$. Then, $(\hx,\ceo(\dd))=(\hx,\ceo(\tx))$ and the
first category is the image of the second category under $\ceo$.

Then, by Corollary \ref{eqcat1}, $\kh$ and $\ceo$ are equivalences
of categories, the one inverse of the other and the diagrama is
commutative.

According to Remark \ref{cuandonueseqcat}, to see that $\nu$ and
$\delta$ are equivalences of categories, we just have to check that
any of them is full.

Let us see that $\delta$ is full. Let $\xid$ and $\xpid$ be
metrizable compactifications packs, let $\phi:X\rightarrow X'$ be a
continuous map. Consider the compactification packs
$(X,X\times(0,1],X\times[0,1])$ and
$(X',X'\times(0,1],X'\times[0,1])$ and the maps $h:X\rightarrow
X\times\{0\}$, $x\rightarrow (x,0)$ and $h':X'\rightarrow
X'\times\{0\}$, $x'\rightarrow (x',0)$.

Observe that the map $\widetilde\phi_0:X\times[0,1]\rightarrow
X'\times[0,1]$, $(x,t)\rightarrow (\phi(x),t)$ is asymptotically
continuous. By Lemma \ref{tecniconuespleno}, there exists
asymptotically continuous functions $\hf:\tx\rightarrow
X\times[0,1]$ and $\tgg:X'\times[0,1]\rightarrow \txp$ such that
$\tf|_X=h$ and $\tgg|_{X'\times\{0\}}=h'^{-1}$.

Consider the asymptotically continuous map
$\widetilde\phi=\tgg\circ\widetilde\phi_0\circ\tf$. Fix $x\in X$.
Then,
$\widetilde\phi(x)=\tgg(\widetilde\phi_0(\tf(x)))=\tgg(\widetilde\phi_0(x,0)=\tgg(\phi(x),0)=\phi(x)$.
Then, $\widetilde\phi|_X=\phi$ and $\delta$ is full.\end{proof}

\begin{rmk} The fact that $\nu$ is a equivalence of categories in
this case, is proved independently in \cite{niya} and
\cite{moreno,moreno2}. As immediate a corollary, in this case, two
spaces are coarse equivalent if and only if they have homeomorphic
coronas. It is proved also independently in
\cite{grv,grv2}.\end{rmk}

Despite of this theorem is more general that Theorem 2 of \cite{zs},
because there the authors work just with the Z-sets in the Hilbert
Cube, there the theorem is stronger, because they work with
continuous and coarse maps. Using an argument of that Theorem,
Theorem 2 of \cite{zs} can be deduced form Theorem
\ref{equivmetrizable} here. But will give the theorem in a different
context.

\begin{teo}\label{equivmetrizablehc}

Consider the category of the complements of Z-sets in the Hilbet
cube $Q$ of the finite dimensional cube $[0,1]^n$, with $n\geq 1$,
with the $C_0$ coarse spaces attached metric of the Hilbert cube or
the finite dimensional cube respectively, with the coarse and
continuous maps identifying two when they are close (represented by
$(\hx,\ce)$), the category of the metrizable compactification packs
$\xid$ such that $\tx$ is the Hilbert cube or the finite dimensional
cube with the continuous and asymptotically continuous maps
identifying two when they are equal in the corona (represented by
$\xid$) and the category of the compact Hausdorff spaces with the
continuous maps (represented by $X$). Consider the functors
described on Remarks \ref{functorcompletoc0},
\ref{functorcompletohigsonroe} and \ref{functorcompletonuydelta}.

Then, $\ceo$, $\hx$, $\nu$ and $\delta$ are coarse equivalences,
being $\ceo$ and $\hx$ the one inverse of the other, and  the
following diagram is commutative:
$$\xymatrix@C=21mm@R=21mm{(\hx,\ce) \ar@{->}@<-1.2ex>[r]_{\kh}^{\hspace{-15pt}\,\,\,\,\,\,\,\,\cong}
\ar@{->}^{\,\nu\,\,\,\,\,\,\,\,\,\,\,\,\,\,\,\circlearrowright}|{\,\cong\,}
[rd] &\xiv\ar@{->}@<-1.2ex>[l]_{\ceo}
\ar@{->}[d]_{\delta\,\,}|{\,\cong\,}\\&X }$$

\end{teo}

\begin{proof}[Proof of Theorem \ref{equivmetrizablehc}]From Theorem \ref{equivmetrizable} and the categories
are defined, we deduce that $\ceo$ and $\kh$ are coarse
equivalences, the one inverse of the other and the diagram is
commutative.

That $\delta$ is a dense functor follows from last paragraph of
Section  \ref{introduccionzsets}. That it is faithful, from Theorem
\ref{equivmetrizable}.

Let us see that $\delta$ is full. Take $\xid$, $\xpid$ with
$\tx\in\{Q\}\cup\{[0,1]^n\}_{n=1}^\infty$ and such that $X$ and $X'$
are Z-sets of $\tx$ and $\txp$ respectively. Consider a continuous
map $f:X\rigtharrow X'$. In case $\tx=\txp=Q$, the section ``T is
full'' of Theorem 2 of \cite{zs}'s proof (pag. 5235, below), they
define a continuous extension of $f$ to $\tf:\hx\rightarrow\txp$
such that $\tf(\hx)\subset\hxp$, using the fact that $X$ is a Z-set
and that $Q$ is an AR and some properties described on
\ref{introduccionzsets}. But this argument is valid in the finite
dimensional cube case, so we have a continuous extension
$\tf:\hx\rightarrow\txp$ with $\tf(\hx)\subset\hxp$ and
$\tf(X)\subset X')$ in every case. By Remark \ref{contasimptcont},
$\tf$ is asymptotically continuous. Then, $\delta$ is full.

Therefore, $\delta$ is a equivalence of categories. Since
$\nu=\delta\circ\kh$, $\nu$ is a equivalence of
categories.\end{proof}

\begin{rmk} Theorem 8 of \cite{zs} (pag. 5238), can be easily proven
using Theorem \ref{equivmetrizablehc}.\end{rmk}

\section*{Acknowledgement}

The author thanks very specially his advisor, Manuel Alonso Morón.
The results of this article, comes from the Author's PhD Thesis
\cite{moreno2,moreno}, and are suggested and oriented by his advisor
from the beginning of this thesis until the publication of the
present article.

He also thanks José María Higes, Jerzy Dydak for his helpful
discusions, to John Roe and Alexander Dranishnikov for their
orientations in the author's stays at Penn State University and
University of Florida, respectively and to Department of Geometría y
Topología of Universidad Complutense de Madrid for his support.

\end{document}